\documentclass[preprint,12pt,3p]{elsarticle}
\usepackage{verbatim}
\usepackage{amsmath,amsthm,amsfonts,amssymb}
\usepackage{algorithm}
\numberwithin{equation}{section}
\usepackage{natbib}
\usepackage{algorithmicx}
\def\BState{\State\hskip-\ALG@thistlm}
\makeatother
\usepackage{algpseudocode}
\newcommand\NoDo{\renewcommand\algorithmicdo{}}
\usepackage{mathtools}
\usepackage{comment}
\usepackage{cases}
\newcommand{\Break}{\State \textbf{break} }
\newcommand{\kronecker}{\raisebox{1pt}{\ensuremath{\:\otimes\:}}}
\usepackage{hyperref}
\usepackage[normalem]{ulem}
\usepackage{xcolor,sectsty}
\usepackage{subfigure}
\usepackage[active,new,old]{correct} 
\newtheorem{theorem}{Theorem}[section]
\newtheorem{lemma}[theorem]{Lemma}
\newtheorem{definition}[theorem]{Definition}
\newtheorem{example}[theorem]{Example}
\newtheorem{remark}[theorem]{Remark}
\newtheorem{corollary}[theorem]{Corollary}
\newtheorem{preposition}[theorem]{Preposition}

\journal{Comput. Math. Appl.}

\newcommand{\mc}[1]{\mathcal {#1}}
\newcommand{\dg}{{\dagger}}
\newcommand{\n}{{*_N}}

\newcommand{\m}{{*_M}}

\begin{document}

\begin{frontmatter}

\title{ 
{\bf Further results on the Drazin inverse of even-order tensors}}

\author{ Ratikanta Behera $^*$, Ashish Kumar Nandi $^\dag$$^a$,  Jajati Keshari Sahoo $^\dag$$^b$}

\address{ 
          $^{*}$ Department of Mathematics and Statistics,\\
Indian Institute of Science Education and Research Kolkata,\\
 Nadia, West Bengal, India.\\
\textit{E-mail}: \texttt{ratikanta@iiserkol.ac.in}
                        
                        \vspace{.3cm}
                        
               $^{\dag}$ Department of Mathematics,\\
                       BITS Pilani, K.K. Birla Goa Campus, Goa, India
                        \\\textit{E-mail$^a$}: \texttt{ashish.nandi123\symbol{'100}gmail.com }
                        \\\textit{E-mail$^b$}: \texttt{jksahoo\symbol{'100}goa.bits-pilani.ac.in}
         
                        }

\begin{abstract}
The notion of the Drazin inverse of an even-order tensor with the Einstein product was introduced, very recently [J. Ji and Y. Wei. Comput. Math. Appl., 75(9), (2018), pp. 3402-3413]. In this article, we further elaborate this theory by producing  a few characterizations of the Drazin inverse and the W-weighted Drazin inverse of tensors.  In addition to these, we compute the Drazin inverse of tensors using different types of generalized inverses and full rank decomposition of tensors. We also address the solution to the multilinear systems using the Drazin inverse  and iterative (higher order Gauss-Seidel) method of tensors.  Besides this,  the convergence analysis of the iterative technique is also investigated within the framework of the Einstein product.

\end{abstract}


\end{frontmatter}

\section{Introduction}
\label{sec1}

The Drazin inverse plays an important role in various applications in singular differential \cite{liang2019} and difference equations \cite{meyer}, Markov chains \cite{camp,meyer}, investigation of Cesaro-Neumann iterations \cite{hart82}, Cryptography \cite{levin}, and  iterative methods \cite{miller,wei98}. Specifically, the Drazin inverse extensively used to solve the system of linear equations, where the iterative schemes lead from matrix splitting.  However, many interesting physical systems are required to store huge volumes of multidimensional data, and in recognition of potential modeling accuracy, matrix representation of data analysis is not enough to represent all the information. Tensors are natural multidimensional generalizations of matrices, which efficiently solve these problems. In this context, Ji and Wei \cite{Wei18} introduced the Drazin inverse of an even-order tensor through the core-nilpotent decomposition to solve singular tensor equations. It will be more applicable if we study the characterization of the Drazin inverse of tensors, and hence this inverse of tensors will open different paths in the above areas.

On the other hand, the concept of tensor-structured numerical methods have opened new perspectives for solving multilinear systems, recently.   Many computational and theoretical problems require different types of generalized inverses when a tensor is singular or arbitrary order. The authors of \cite{bral} discussed the representations and properties of the ordinary tensor inverse and introduced tensor-based iterative methods to solve high-dimensional Poisson problems in the multilinear system framework. This interpretation is extended to the Moore-Penrose inverse of tensors in \cite{bm, liang2019} and discussed the solution of multilinear systems and tensor nearness problem associated with tensor equations. Using such theory of Einstein product, Liang et al. \cite{liang} investigated necessary and sufficient conditions for the invertibility of tensors, and proposed the LU and the Schur decompositions of a tensor. Further, Stanimirovic et al. \cite{stan} introduced some basic properties of the range and null space of tensors, and the adequate definition of the tensor rank (i.e., reshaping rank). In view of reshape rank, Behera et al. \cite{behera18} discussed full rank decomposition of tensors  via Einstein product. The vast work on the generalized inverse of tensors \cite{bm,  ji2017, jin2017, stan, sun} and its applications to the solution of multilinear systems \cite{bral, Wei18}, motivate us to study the characterizations of the Drazin inverse, W-weighted Drazin inverse and iterative technique in the framework of tensors.

In this paper, we further study the Drazin inverse of tensors. This study can lead to the enhancement of the computation of the  Drazin inverse of tensors along with solutions of multilinear structure in multidimensional systems. In this regard, we discuss different characterizations of the Drazin inverse and W-weighted Drazin inverse of tensors.  In addition to these, some new methods for computing the  Drazin inverse of a tensor is proposed. Since the reduction of spatial dimensions and the generalized inverse of tensors needs to solve tensor-based partial differential equations, here we concentrate on the tensor iterative method (higher order Gauss-Seidel) and its convergence analysis using the theory of Einstein product.

\subsection{Organization of the paper}
The rest of the paper is organized as follows. 
In Section 2, we discuss some notations and
definitions, which are the necessary ingredient for proving the main results in Sections 3-5. Several characterizations of the Drazin inverse are discussed in Section 3. Besides these, we have computed the Drazin inverse of  tensors with the help of other generalized inverses. The notion of W-weighted Drazin inverse and a few properties of this inverse are introduced in Section 4. Then taking advantage of the Drazin inverse of tensors, we discuss the solution of multilinear systems in Section 5. In addition to these, the convergence analysis of the iterative technique is also investigated within the framework of the Einstein product. In Section 6, we conclude this paper with some remarks.

\section{Preliminaries}
\subsection{Some notations and definitions}

For convenience, we first briefly explain some of the terminologies which will be used here onwards.  Let $\mathbb{R}^{J_1\times J_2\times \cdots\times J_N }(\mathbb{C}^{J_1\times J_2\times \cdots\times  J_N)}$  be the set of  order $N$ and dimension $J_1 \times J_2 \times \cdots \times J_N$ tensors over the real (complex) field $\mathbb{R} (\mathbb{C})$, where $J_1, J_2, \cdots , J_N$ are positive integers. An order $N$ tensor is denoted as $\mc{A} = (a_{j_1,j_2,\cdots ,j_N})$. Note that throughout the paper, tensors are represented in calligraphic letters like  $\mc{A}$, and the notation $(\mc{A})_{j_1,j_2, \cdots , j_N}= a_{j_1,j_2, \cdots , j_N}$ represents the scalars. We use some additional notations to simplify our representation,
$$\textbf{j}(N) = \{ j_1, j_2, \cdots, j_N ~|~ 1 \leq j_k \leq J_k, k = 1, 2, \cdots, N\},~~ and ~~~\textbf{J}(N) = J_1 \times J_2 \times \cdots \times J_N.$$ 
Further, we denote $\hat{\textbf{J}}(N) = \{ J_1, J_2, \cdots, J_N\}$.  Let the tensor $ \mc{A} \in \mathbb{C}^{I_1\times I_1\times \cdots\times I_N\times J_1\times J_2\times \cdots\times J_N }$.  In connection with these notations,  the tensor $\mc{A} = (a_{i_1,i_2,\cdots i_N, j_1,j_2,\cdots,j_N})_{1\leq i_k \leq I_k, 1\leq j_l \leq J_l }, k=1,2,\cdots, N$ and $l=1,2,\cdots, N$, is denoted by $\mc{A} = (a_{\textbf{i}(N),\textbf{j}(N)})$. The Einstein product \cite{ein} $ \mc{A}\n\mc{B} \in \mathbb{C}^{\textbf{I}(N)\times\textbf{J}(N)}$ of tensors $\mc{A} \in \mathbb{C}^{\textbf{I}(N)\times\textbf{K}(N)}$ and $\mc{B} \in
\mathbb{C}^{\textbf{K}(N)\times\textbf{J}(N)}$  is defined
by the operation via $\n$
\begin{equation*}\label{Eins}
(\mc{A}\n\mc{B})_{\textbf{i}(N),\textbf{j}(N)}
 =\displaystyle\sum_{k_1...k_N}a_{i_1...i_N{k_1...k_N}}b_{{k_1...k_N}{j_1...j_N}} = \displaystyle\sum_{\textbf{k}(N)} a_{\textbf{i}(N), \textbf{k}(N)}b_{\textbf{k}(N), \textbf{j}(N)}.
\end{equation*}
In particular, if $\mc{B} \in \mathbb{C}^{\textbf{K}(N)}$, then $\mc{A}\n\mc{B} \in \mathbb{C}^{\textbf{I}(N)}$ and 
\begin{equation*}
(\mc{A}\n\mc{B})_{\textbf{i}(N)} = \displaystyle\sum_{{\textbf k}(N)}
a_{\textbf{i}(N), \textbf{k}(N)}b_{\textbf{k}(N)}.
\end{equation*}
The Einstein product is used in the study of the theory of relativity \cite{ein} and in the area of continuum mechanics (\cite{lai}). Using such theory of Einstein product, the range space and null space of a tensor  $\mc{A}\in \mathbb{C}^{\textbf{I}(M)\times\textbf{J}(N)}$ was introduced in \cite{Wei18, stan}, as follows.

$$
\mathfrak{R}(\mc{A}) = \left\{\mc{A}\n\mc{X}:~\mc{X}\in\mathbb{R}^{\textbf{J}(N)}\right\}\mbox{ and } \mc{N}(\mc{A})=\left\{\mc{X}:~\mc{A}\n\mc{X}=\mc{O}\in\mathbb{R}^{\textbf{I}(M)}\right\},
$$
where the tensor $\mc{O}$ denotes the {\it zero tensor}, i.e., all the entries are zero. As a consequence of the definition of the range space and null space of a tensor, it is clear that  $\mc{N}(A)$ is a subspace of $\mathbb{C}^{\textbf{J}(N)}$ and $\mathfrak{R}(\mc{A})$ is a subspace of $\mathbb{C}^{\textbf{I}(M)}$. The relation of range space and some properties of range and null spaces are discussed in \cite{Wei18,stan}. Here, we collect some known results which will be  used in this paper.
\begin{lemma}[Lemma 2.2, \cite{stan}]\label{range-stan}
Let  $\mc{A}\in \mathbb{R}^{\textbf{I}(M)\times \textbf{J}(N)}$,  $\mc{B}\in \mathbb{R}^{\textbf{I}(M)\times \textbf{K}(L)}.$ Then $\mathfrak{R}(\mc{B})\subseteq\mathfrak{R}(\mc{A})$ if and only if there exists  $\mc{U}\in \mathbb{R}^{\textbf{J}(N)\times \textbf{K}(L)}$ such that
$\mc{B}=\mc{A}\n\mc{U}.$
\end{lemma}

Adopting the definition of range space and null space,  Ji and Wei \cite{Wei18} discussed the index of a tensor, as follows.
\begin{definition}{\cite{Wei18}}
Let $\mc{A} \in \mathbb{C}^{\textbf{I}(N) \times \textbf{I}(N)},$ and k be the smallest nonnegative integer  such that, $R(\mc{A}^k)= R(\mc{A}^{k+1}).$ Then $k$ is called the index of $\mc{A}$ and denoted by $ind(\mc{A}).$ 
\end{definition}

We now move to the definition of the Drazin inverse of a tensor, which was studied  in Theorem 3.3 \cite{Wei18}  in the context of range space of a tensor, as follows.  
\begin{definition}\label{dz1.1}  Let $\mc{A} \in \mathbb{C}^{\textbf{I}(N) \times \textbf{I}(N)}$ be a tensor with $ind(\mc{A}) = k$. The tensor $\mc{X} \in \mathbb{C}^{\textbf{I}(N) \times \textbf{I}(N)} $ satisfying the following three tensor equations:
\begin{eqnarray*}
&&(1)~\mc{A}^{k+1}\n\mc{X} = \mc{A}^{k},\\
&&(2)~\mc{X}\n\mc{A}\n\mc{X} = \mc{X},\\
&&(3)~\mc{A}\n\mc{X} = \mc{X}\n\mc{A},
\end{eqnarray*}
is called  the Drazin inverse of $\mc{A}$, and is
denoted by $\mc{A}^{D}$. Specifically, when k = 1, $\mc{A}^D$ is called
the group inverse of $\mc{A}$ and denoted by $\mc{A}^\#$.
\end{definition}

At the same time, the authors of \cite{Wei18} discussed the existence of the Drazin inverse of a tensor. In view of this, we have the following result for the uniqueness of the Drazin inverse of a tensor. 
\begin{theorem}\label{mr1.1}
Let $\mc{A} \in \mathbb{C}^{\textbf{I}(N) \times \textbf{I}(N)}$ be a tensor with $ind(\mc{A}) = k$, then the Drazin inverse $\mc{A}^{D}$ is unique.
\end{theorem}
\begin{proof}
Suppose $\mc{X}$ and $\mc{Y}$ are two Drazin inverses of $\mc{A}$. Now 
\begin{eqnarray}\label{unique1} 
   \mc{X} = \mc{X}\n\mc{A}\n\mc{X}
          = \mc{X}\n\mc{A}\n\mc{X}\n\mc{A}\n\mc{X}
          = \mc{X}\n\mc{X}\n\mc{A}\n\mc{A}\n\mc{X}
          = \mc{X}^{2}\n\mc{A}^{2}\n\mc{X}.  
\end{eqnarray}
Repeating $k$ times the Eq. \eqref{unique1}, we obtain 
\begin{eqnarray*}
\mc{X} &=& \mc{X}^{k+1}\n\mc{A}^{k+1}\n\mc{X}
       = \mc{X}^{k+1}\n\mc{A}^{k}
       = \mc{X}^{k+1}\n\mc{A}^{k}\n\mc{Y}\n\mc{A}\\
       &=& \mc{X}^{k+1}\n\mc{A}^{k}\n\mc{Y}\n\mc{A}\n\mc{Y}\n\mc{A}
       = \mc{X}^{k+1}\n\mc{A}^{K+1}\n\mc{Y}^{2}\n\mc{A}\\
       &=&  \mc{A}^{k+1}\n\mc{X}^{k+1}\n\mc{Y}^{2}\n\mc{A}
       = \mc{A}\n\mc{X}\n\mc{Y}\n\mc{A}\n\mc{Y}\\
       &=& \mc{A}^{2}\n\mc{X}\n\mc{Y}^{2}=\mc{A}^{2}\n\mc{X}\n\mc{Y}\n\mc{Y}.
\end{eqnarray*}
Now substituting $\mc{Y}$ as $\mc{Y}\n\mc{A}\n\mc{Y}$ and repeating it $k-2$ times, one can obtain $\mc{X} = \mc{Y}$.   
\end{proof}
In connection with range space and null space,  Ji and Wei \cite{Wei18} discussed the characterization of the Drazin inverse of tensors, as follows.

\begin{theorem}[Theorem 3.4, \cite{Wei18}]\label{wei1.1}
Let $ \mc{A} \in \mathbb{C}^{\textbf{I}(N) \times \textbf{I}(N)}$ and ind$(\mc{A})=k.$ Then for $l\geq k,$ the following holds
\begin{enumerate}
    \item[(a)] $\mathfrak{R}(\mc{A}^{D}) = \mathfrak{R}(\mc{A}^{k})$ and $\mc{N}(\mc{A}^{D}) = \mc{N}(\mc{A}^{k}),$
    \item[(b)] $\mathfrak{R}(\mc{A}\n\mc{A}^{D}) = \mathfrak{R}(\mc{A}^{k})$, $\mc{N}(\mc{A}\n\mc{A}^{D}) = \mc{N}(\mc{A}^{k})$, $\mathfrak{R}(I-\mc{A}\n\mc{A}^{D}) =\mc{N}(\mc{A}^{k})$ and $\mc{N}(I-\mc{A}\n\mc{A}^{D}) = \mathfrak{R}(\mc{A}^{k})$.
\end{enumerate}
\end{theorem}

\begin{theorem}[Theorem 3.2, \cite{Wei18}]\label{weitm3.2}
Let  $\mc{A}\in \mathbb{R}^{\textbf{I}(N)\times \textbf{I}(N)}$. If ind$(\mc{A}) = k$, then $\mathfrak{R}(\mc{A}^l) = \mathfrak{R}(\mc{A}^k)$ and $\mc{N}(\mc{A}^l) = \mc{N}(\mc{A}^k)$ for any positive integer $l\geq k$.
\end{theorem}

We next present the definition of the diagonal tensor which was introduced earlier in \cite{bral,sun}.
\begin{definition} [Diagonal tensor, Definition 3.12,  \cite{bral}] A tensor $\mc{D}
 =(d_{\textbf{i}(n),\textbf{j}(n)})
  \in \mathbb{R}^{\textbf{I}(N) \times \textbf{I}(N)} $ is
   called a {\it diagonal
   tensor} if all its entries are zero except for $d_{\textbf{i}(n), \textbf{i}(n)}.$
\end{definition}

The definition of an {\it upper off-diagonal tensor} and {\it lower off-diagonal tensor } are defined under the influence of Definition 3.12, \cite{sun}, as follows.

\begin{definition} [Upper off-diagonal tensor] A tensor $\mc{U}=(u_{{\textbf{i}(N)},{\textbf{j}(N)}}) \in  \mathbb{R}^{\textbf{I}(N) \times \textbf{I}(N)}$ is called an {\it upper off-diagonal tensor} if all entries below the main diagonal are zero,  i.e.,  $u_{{\textbf{i}(N)},{\textbf{j}(N)}}$ = 0 for $j_k < i_k$, where $k=1,2,\cdots N.$
\end{definition}
\begin{definition} [Lower off-diagonal tensor]
A tensor $\mc{L}= (l_{{\textbf{i}(N)},{\textbf{j}(N)}})\in  \mathbb{R}^{\textbf{I}(N) \times \textbf{I}(N)}$ is called a {\it lower off-diagonal tensor} if all entries above the main diagonal are zero,  i.e.,  $l_{{\textbf{i}(N)},{\textbf{j}(N)}}$ = 0 for $i_k < j_k$, where $k=1,2,\cdots N.$
\end{definition}

Using the notation of diagonal tensor we define diagonal dominant tensor, as follow.
\begin{definition} [Diagonally dominant tensor]
A tensor $\mc{A} = (a_{{\textbf{i}(N)},{\textbf{j}(N)}})\in  \mathbb{R}^{\textbf{I}(N) \times \textbf{I}(N)}$ is called diagonally dominant if 
\begin{equation}
    |a_{{\textbf{i}(N)},{\textbf{i}(N)}}| \geq
        \displaystyle\sum_{\substack{\textbf{j}(N)\neq \textbf{i}(N)\\\textbf{j}(N)}}
              |a_{{\textbf{i}(N)},{\textbf{j}(N)}}|. 
\end{equation}
 \end{definition}

We recall the definition of an eigenvalue of a tensor as below.
\begin{definition}[Eigenvalue of a tensor, Definition 2.3, \cite{liang2019}] \label{defeigen}
Let $\mc{A}\in\mathbb{C}^{\textbf{I}(N)\times\textbf{I}(N)}.$ A complex number $\lambda\in\mathbb{ C}$ is called an eigenvalue of $\mc{A}$ if there exist some nonzero tensor $\mc{X}\in\mathbb{C}^{\textbf{I}(N)}$ such that $\mc{A}\n\mc{X}=\lambda\mc{X}.$
\end{definition}
The nonzero tensor $\mc{X}$ is called eigen vector of $\mc{A}$ and we define the spectral radius $\rho(\mc{A})$ of $\mc{A},$ be the largest absolute value of the eigenvalues of $\mc{A}$.
As a consequence of the definition of eigenvalue, the following lemma easily holds.
\begin{lemma}\label{eigenpower}
Let $\mc{A}\in\mathbb{C}^{\textbf{I}(N)\times\textbf{I}(N)}.$ If $\lambda$ is an eigenvalue of $\mc{A},$ then for $m\in\mathbb{N}$, $\lambda^m$ is an eigenvalue of $\mc{A}^m.$
\end{lemma}
Let $\mc{A} = (a_{\textbf{i}(N), \textbf{j}(N)}) \in\mathbb{C}^{\textbf{I}(N)\times \textbf{I}(N)}$. Then  we have $\displaystyle\lim_{k\rightarrow\infty}\mc{A}^k =   \displaystyle\lim_{k\rightarrow\infty} \left[ (\mc{A}^k)_{\textbf{i}(N), \textbf{j}(N)}\right].$ 
In view of this fact, we next present the definition of the convergence of a tensor. 
\begin{definition} [Convergent tensor]  A tensor $\mc{A} \in\mathbb{C}^{\textbf{I}(N)\times\textbf{I}(N)}$ is called convergent tensor if $\mc{A}^k \rightarrow \mc{O}$ as $k \rightarrow \infty$.
\end{definition}

We now introduce the definition of convergence of a power series of tensor, which is a generalization of the power series in matrices \cite{makar}.
\begin{definition} [Tensor series convergent]\label{tenseries}
Let $\mc{A}\in\mathbb{C}^{\textbf{I}(N)\times\textbf{I}(N)}.$ The series $\displaystyle\sum_{k=0}^\infty c_k\mc{A}^{k}$ is convergent if $\displaystyle\sum_{k=0}^\infty c_k(\mc{A}^k)_{\textbf{i}(N),\textbf{j}(N)}$ is convergent for every $\textbf{i}(N)$ and $\textbf{j}(N)$.
\end{definition}

\subsection{Reshape rank and decomposition of a tensor}
The reshape operation systematically rearranges the entries of an arbitrary order-tensor into a matrix \cite{stan}. This operation is denoted by {\it rsh}, and implemented by means of the standard Matlab function {\it reshape}. 
\begin{definition}[Definition 3.1, \cite{stan}]
The 1-1 and onto reshape map, rsh,  is defined as $rsh :  \mathbb{C}^{\textbf{I}(M) \times \textbf{J}(N)} \longrightarrow \mathbb{C}^{\mathfrak{M} \times \mathfrak{N}}$ with
\begin{equation}
    rsh(\mc{A}) = A = reshape(\mc{A},\mathfrak{M}, \mathfrak{N}),
\end{equation}
      where $ \mc{A} \in \mathbb{C}^{\textbf{I}(M) \times \textbf{J}(N)}, $  the matrix $ A \in  \mathbb{C}^{\mathfrak{M} \times \mathfrak{N}},$ $\mathfrak{M}=\prod_{i=1}^M I_i$ and $\mathfrak{N}=\prod_{j=1}^N J_j$. Further, the inverse reshaping is the mapping defined as $rsh^{-1} :  \mathbb{C}^{\mathfrak{M} \times \mathfrak{N}} \longrightarrow  \mathbb{C}^{\textbf{I}(M) \times \textbf{J}(N)}$ with 
\begin{equation}
rsh^{-1}(A) = \mc{A} = reshape(A,I_1,\cdots,I_M,J_1,\cdots,J_N),
 \end{equation}
 where  the matrix $ A \in  \mathbb{C}^{\mathfrak{M} \times \mathfrak{N}}$ and the tensor $\mc{A} \in \mathbb{C}^{\textbf{I}(M) \times \textbf{J}(N)}$. 
\end{definition}

Further, Lemma 3.2 in \cite{stan} defined the rank of a tensor $\mc{A} $, denoted by $ rshrank(\mc{A}) $, as 
\begin{equation}\label{21}
    rshrank(\mc{A}) = rank(rsh(\mc{A})).  
\end{equation}

Adopting the reshaping operation, Behera et al. \cite{behera18} defined the Moore-Penrose inverse of an arbitrary order tensor. Whereas, the authors of \cite{sun} was introduced this Moore-Penrose inverse for even-order tensors, which is recalled next.

\begin{definition}[Definition 2.2, \cite{sun}]\label{defmpi}
Let $\mc{A} \in \mathbb{C}^{\textbf{I}(N) \times
\textbf{I}(N)}$. The tensor $\mc{X} \in
\mathbb{C}^{\textbf{I}(N) \times \textbf{I}(N)} $ satisfying the following four tensor equations:
$$\mc{A}\n\mc{X}\n\mc{A} = \mc{A};~~ \mc{X}\n\mc{A}\n\mc{X} = \mc{X}; ~~
(\mc{A}\n\mc{X})^* = \mc{A}\n\mc{X};~~ (\mc{X}\n\mc{A})^* = \mc{X}\n\mc{A} $$
is called  the \textbf{Moore-Penrose inverse} of $\mc{A}$, and is
denoted by $\mc{A}^{\dg}$. 
In particular, if the tensor $\mc{X}$ satisfies only first equation, then $\mc{X}$ is called $\{1\}$-inverse of $\mc{A}$ and denoted by $\mc{A}^{(1)}.$
\end{definition}
On the other hand, using reshape rank of a tensor $\mc{A}$, Behera et al. \cite{behera18} discussed full rank decomposition of a tensor, as stated below. 

\begin{theorem}\label{49}[Theorem- 2.22 \cite{behera18}]
Let $ \mc{A} \in  \mathbb{C}^{\textbf{I}(N) \times \textbf{I}(N)} $. Then there exist a left invertible tensor 
$ \mc{F} \in \mathbb{C}^{\textbf{I}(N) \times \textbf{H}(R)} $ and a right invertible tensor $ \mc{G} \in \mathbb{C}^{\textbf{H}(R) \times \textbf{I}(N)} $  such that
\begin{equation} \label{37}
    \mc{A} = \mc{F} *_R \mc{G},
\end{equation}
where  $ rshrank(\mc{F}) =rshrank(\mc{G})= rshrank(\mc{A}) = r = \textbf{H}(R)$.  
\end{theorem}
In connection with the Moore-Penrose inverse of tensors, the singular value decomposition (SVD) discussed in Lemma 3.1  \cite{sun} for a complex tensor. However, the authors of
\cite{bral} proved the same result for a real tensor. 
\begin{lemma}{(Lemma 3.1, \cite{sun})}\label{SVDTensor}
 A tensor $\mc{A} \in
\mathbb{C}^{\textbf{I}(N) \times \textbf{J}(N)}$
 can be decomposed  as $$\mc{A} = \mc{U}*_N\mc{B}*_N\mc{V}^*,$$
 where $\mc{U} \in \mathbb{C}^{\textbf{I}(N) \times \textbf{I}(N)}$ and
 $\mc{V} \in \mathbb{C}^{\textbf{J}(N) \times \textbf{J}(N)}$ are unitary
 tensors, and
 $\mc{B} \in \mathbb{C}^{\textbf{I}(N) \times \textbf{J}(N)}$ is a
 tensor such that
 $(\mc{B})_{{\textbf{i}(N)},{\textbf{j}(N)}} =0$, if  $i_k \neq  j_k$, where $k=1,2,\cdots N.$
 \end{lemma}

\section{Results on Drazin inverse of tensors}

The Drazin inverse of tensors plays nearly the same role as the standard inverse of an invertible tensor. General properties of the Drazin inverse of tensors via Einstein product can be found in \cite{Wei18}. In this section, we further embellish on this theory by producing a few more characterizations of this inverse. We divided this section into two parts. In the first part, we obtain several identities involving the Drazin inverse of tensors. The second part contains the computation of the Drazin inverse of tensors.
\subsection{Some identities}\label{subsec1}
It is worth mentioning that Ji and Wei \cite{Wei18} studied the Drazin inverse of tensors, which motivates us to investigate further on the theory of the Drazin inverse of tensors.  We find some interesting identities. Some of these are used in the next sections. The very first result of this section will numerously use in other consequential identities. 

\begin{lemma}\label{mr11.2}
Let $\mc{A} \in \mathbb{C}^{\textbf{I}(N) \times \textbf{I}(N)}$ be a tensor with $ind(\mc{A}) = k$. Then 
$\mc{A}^{p}\n(\mc{A}^{D})^{p} = \mc{A}^{D}\n\mc{A} = (\mc{A}^{D})^{p}\n\mc{A}^{p} \textnormal{~~~for every positive integer} ~ p$.
\end{lemma}

\begin{proof}
\begin{eqnarray*}
\mc{A}^{p}\n(\mc{A}^{D})^{p}&=&
\underbrace{\mc{A}\n\cdots \n \mc{A}}_{p\rm\ times}\n\underbrace{\mc{A}^{D}\n\cdots \n \mc{A}^{D}}_{p\rm\ times}\\
&=& \underbrace{\mc{A}\n\cdots \n \mc{A}}_{(p-1)\rm\ times}
\n(\mc{A}^{D}\n \mc{A}\n \mc{A}^{D})\n \underbrace{\mc{A}^{D}\n \cdots \n \mc{A}^{D}}_{(p-2)\rm\ times}\\
&=&\underbrace{\mc{A}\n\cdots \n \mc{A}}_{(p-1)\rm\ times}\n\underbrace{\mc{A}^{D}\n\cdots \n \mc{A}^{D}}_{(p-1)\rm\ times}\\
&=& \cdots ~~ \cdots \\  
&=&\mc{A}\n\mc{A}^{D} = \mc{A}^{D}\n\mc{A}.
\end{eqnarray*}
\end{proof}
Using the method as in the proof of the above Lemma \ref{mr11.2}, one can show the next theorem.
\begin{theorem}\label{mr1.10}
Let $\mc{A} \in \mathbb{C}^{\textbf{I}(N) \times \textbf{I}(N)}$ be a tensor with $ind(\mc{A}) = k$. Then the following holds
\begin{enumerate}
\item[(a)] $(\mc{A}^{*})^{D} = (\mc{A}^{D})^{*},$
\item[(b)] If $l>m>0$, then $ (\mc{A}^{D})^{l-m} = \mc{A}^{m}\n(\mc{A}^{D})^{l},$
\item[(c)] If $m>0$ and $(l-m)\geq k$, then $(\mc{A})^{l-m} = \mc{A}^{l}\n(\mc{A}^{D})^{m}.$
\end{enumerate}
\end{theorem}

Recall that a  tensor $~\mc{A} \in \mathbb{C}^{\textbf{I}(N) \times \textbf{I}(N)}$ is called nilpotent if $\mc{A}^k=\mc{O},$ where $\mc{O}\in \mathbb{C}^{\textbf{I}(N) \times \textbf{I}(N)}$ is the  zero tensor. It is trivial that,
the nilpotent tensors are always singular. The next result presents the existence of the Drazin inverse of nilpotent tensors. 
\begin{corollary}\label{mr1.9}
Let  $\mc{A} \in \mathbb{C}^{\textbf{I}(N) \times \textbf{I}(N) }$ be a nilpotent tensor with index $k.$ Then $\mc{A}^{D} = \mc{O}$.
\end{corollary}
The power of the  Drazin inverse and the Drazin inverse of power tensors  can be switched without changing the result. Which is discussed in the following theorem.
\begin{theorem}\label{mr1.2}
Let $\mc{A}\in \mathbb{C}^{\textbf{I}(N) \times \textbf{I}(N) }$ and ind$(\mc{A})=k.$ Then for $l\in\mathbb{N},$ the following holds
\begin{enumerate}
\item[(a)] $(\mc{A}^{l})^{D} = (\mc{A}^{D})^{l},$
    \item[(b)] $(\mc{A}^{D})^{\#} = \mc{A}^{2}\n\mc{A}^{D},$
  \item[(c)] $((\mc{A}^{D})^{D})^{D} = \mc{A}^{D}.$
\end{enumerate}
\end{theorem}
\begin{proof}
$(a)$ Let the tensor $\mc{X} = (\mc{A}^{D})^{l}$. It is enough to show $\mc{X}$ is the Drazin inverse of $\mc{A}^l.$ Now 
\begin{eqnarray}\label{351}
(\mc{A}^{l})^{k+1}\n\mc{X} &=& (\mc{A}^{l})^{k}\n\mc{A}^{l}\n(\mc{A}^{D})^{l} = (\mc{A}^{l})^{k}\n\mc{A}\n\mc{A}^{D}\\\nonumber
&=& (\mc{A}^{k})^{l-1}\n\mc{A}^{k}\n\mc{A}\n\mc{A}^{D} = (\mc{A}^{k})^{l-1}\n\mc{A}^{k} = (\mc{A}^{k})^{l} = (\mc{A}^{l})^{k},
\end{eqnarray}
further, 
\begin{eqnarray}\label{352}
\mc{X}\n\mc{A}^{l}\n\mc{X} &=& (\mc{A}^{D})^{l}\n\mc{A}^{l}\n(\mc{A}^{D})^{l} = (\mc{A}^{D})^{l}\n\mc{A}\n\mc{A}^{D}\\\nonumber
&=& (\mc{A}^{D})^{l-1}\n\mc{A}^{D}\n\mc{A}\n\mc{A}^{D} = (\mc{A}^{D})^{l-1}\n\mc{A}^{D} = \mc{X}
\end{eqnarray}
and 
\begin{eqnarray}\label{353}
(\mc{A}^{D})^{l}\n\mc{A}^{l} = \mc{A}^{D}\n\mc{A} = \mc{A}\n\mc{A}^{D} = \mc{A}^{l}\n(\mc{A}^{D})^{l}.
\end{eqnarray}
From \eqref{351}, \eqref{352} and \eqref{353}, we conclude $\mc{X}$ is the Drazin inverse of $\mc{A}^l.$ Hence $(\mc{A}^{l})^{D} = (\mc{A}^{D})^{l}.$\\
$(b)$ Let  $\mc{X} = \mc{A}^{2}\n\mc{A}^{D}$. By Definition \ref{dz1.1}, we have
\begin{eqnarray*}
\mc{A}^{D}\n\mc{X}\n\mc{A}^{D} = \mc{A}^{D}\n\mc{A}^{2}\n\mc{A}^{D}\n\mc{A}^{D} = \mc{A}^{D}\n\mc{A}\n\mc{A}^{D} = \mc{A}^{D},
\end{eqnarray*}
further, 
\begin{eqnarray*}
\mc{X}\n\mc{A}^{D}\n\mc{X} &=& \mc{A}^{2}\n\mc{A}^{D}\n\mc{A}^{D}\n\mc{A}^{2}\n\mc{A}^{D}=\mc{A}\n\mc{A}^{D}\n\mc{A}\n\mc{A}^{D}\n\mc{A}\n\mc{A}^{D}\n\mc{A}\\
&=& \mc{A}\n\mc{A}^{D}\n\mc{A} = \mc{A}^{2}\n\mc{A}^{D} = \mc{X},
\end{eqnarray*}
and 
\begin{eqnarray*}
\mc{A}^{D}\n\mc{X} = \mc{A}^{D}\n\mc{A}^{2}\n\mc{A}^{D} = \mc{A}\n\mc{A}^{D}\n\mc{A}\n\mc{A}^{D} = \mc{A}^{2}\n\mc{A}^{D}\n\mc{A}^{D} = \mc{X}\n\mc{A}^{D}. 
\end{eqnarray*}
Thus $(\mc{A}^{D})^{\#} = \mc{A}^{2}\n\mc{A}^{D}$.\\
$(c)$ Consider $\mc{Y} = \mc{A}^{D}$. Now applying the Definition of the Drazin inverse, we obtain 
\begin{eqnarray*}
((\mc{A}^{D})^{D})^{k+1}\n\mc{Y} &=& ((\mc{A}^{D})^{D})^{k}\n(\mc{A}^{D})^{D}\n\mc{A}^{D}=((\mc{A}^{D})^{D})^{k-1}\n(\mc{A}^{D})^{D}\n\mc{A}^{D}\n(\mc{A}^{D})^{D}\\ 
&=& ((\mc{A}^{D})^{D})^{k-1}\n(\mc{A}^{D})^{D} =  ((\mc{A}^{D})^{D})^{k},    \\
\textnormal{and ~~~}\mc{Y}\n(\mc{A}^{D})^{D}\n\mc{Y} &=& \mc{A}^{D}\n(\mc{A}^{D})^{D}\n\mc{A}^{D}=\mc{A}^{D}\n\mc{A}^{D}\n\mc{A}\n\mc{A}^{D}\n(\mc{A}^{D})^{D}\\
&=& \mc{A}\n(\mc{A}^{D})^{3}\n(\mc{A}^{D})^{D}= \cdots= \mc{A}^{k-1}\n(\mc{A}^{D})^{k+1}\n(\mc{A}^{D})^{D}\\
&=& \mc{A}^{k-1}\n(\mc{A}^{D})^{k}=\mc{A}^{k-1}\n(\mc{A}^{D})^{k-1}\n\mc{A}^{D}\\
&=& \mc{A}\n\mc{A}^{D}\n\mc{A}^{D} = \mc{Y}.
\end{eqnarray*}
Further, we have $(\mc{A}^{D})^{D}\n\mc{A}^{D}=\mc{A}^{D}\n(\mc{A}^{D})^{D}$. Hence $\mc{Y}$ is the Drazin inverse of $(\mc{A}^{D})^{D}$.
\end{proof}

\begin{corollary}\label{mr1.8}
Let $\mc{A} \in \mathbb{C}^{\textbf{I}(N) \times \textbf{I}(N)}$ be a tensor with $ind(\mc{A}) = k$. Then the following holds
\begin{enumerate}
\item[(a)]  $(\mc{A}^{l})^{\#} = (\mc{A}^{D})^{l}$ for $l\geq k,$
\item[(b)] $\mc{A}^{D}\n(\mc{A}^{D})^{\#} = \mc{A}\n\mc{A}^{D}$.
\end{enumerate}
\end{corollary}

\begin{remark}\label{dardar}
Like the well-known property of the Moore Penrose inverse, i.e.,  $(\mc{A}^\dg)^\dg = \mc{A}$ (see Proposition 3.3, \cite{sun}), it is worth pointing out that the Drazin inverse is not following the property, i.e.,  $(\mc{A}^{D})^{D} \neq \mc{A}$, as shown below with an example.
\end{remark}

\begin{example}\label{ex1.1}
Consider a tensor
$~\mc{A}=(a_{ijkl})
 \in \mathbb{R}^{\overline{2\times3}\times\overline{2\times3}}$ with entries 
\begin{eqnarray*}
a_{ij11} =
    \begin{pmatrix}
    0 & 1 & 1 \\
    1 & 1 &  1
    \end{pmatrix},~
a_{ij12} =
    \begin{pmatrix}
     0 & 0 & 1\\
     0 & 1 & 1
    \end{pmatrix},~
a_{ij13} =
    \begin{pmatrix}
     0 & 0 & 0\\
     0 & 0 & 1
\end{pmatrix},\\
a_{ij21} =
    \begin{pmatrix}
     0 & 1 & 1\\
     0 & 1 & 1
    \end{pmatrix},~
    a_{ij22} =
    \begin{pmatrix}
    0 & 0 & 1 \\
    0 & 0 &  1
    \end{pmatrix},~
a_{ij23} =
    \begin{pmatrix}
     0 & 0 & 0\\
     0 & 0 & 0
    \end{pmatrix}.
\end{eqnarray*}
Then one can easily verify  $(\mc{A}^{D})^{D} \neq \mc{A}$, however
$\mc{A}^D=\mc{O} = (\mc{A}^D)^D =((\mc{A}^{D})^{D})^D\in \mathbb{R}^{\overline{2\times3}\times\overline{2\times3}}$.
 \end{example}

At this point one may be interested to know, when does the above equality property (Remark  \ref{dardar}) hold $?$ The answer to this query is discussed in the next theorem.

\begin{theorem}\label{mrr1.1}
Let $\mc{A}\in \mathbb{C}^{\textbf{I}(N) \times \textbf{I}(N) }$. Then $(\mc{A}^{D})^{D} = \mc{A}$ if and only if ind$(\mc{A})=1.$
\end{theorem}

\begin{proof}
Let $\mc{A}$ be of index $1$. Now $(\mc{A}^{D})^{k+1}\n\mc{A} =  (\mc{A}^{D})^{k-1}\n\mc{A}^{D}\n\mc{A}\n\mc{A}^{D} = (\mc{A}^{D})^{k}$. Since $\mc{A}$ is of index $1$ so $\mc{A}\n\mc{A}^{D}\n\mc{A} = \mc{A}\n\mc{A}^{\#}\n\mc{A} = \mc{A}$ and $\mc{A}\n\mc{A}^{D} = \mc{A}^{D}\n\mc{A}$. Conversely, let
$(\mc{A}^{D})^{D} = \mc{A}$. By definition of the Drazin inverse, $\mc{A}\n\mc{A}^{D}\n\mc{A} = \mc{A}$. This implies $\mc{A} = \mc{A}^{2}\n\mc{A}^{D}$. So by Lemma \ref{range-stan}, $\mathfrak{R}(\mc{A})\subseteq \mathfrak{R}(\mc{A}^2).$ It is obvious that $\mathfrak{R}(\mc{A}^2)\subseteq \mathfrak{R}(\mc{A}). $ Thus $\mathfrak{R}(\mc{A}^2) =\mathfrak{R}(\mc{A})$. Hence $\mc{A}$ is of index $1$. 
\end{proof}
Thus, the special case of the Drazin inverse (index $k$ = 1) gives necessary and sufficient conditions for the equality.

\begin{remark}\label{reverse}
The Theorem \ref{mr1.2} (a) is not true if we use two different tensor $\mc{A}$ and $\mc{B}$, i.e.,  $(\mc{A}\n\mc{B})^{D} \neq \mc{A}^{D}\n\mc{B}^{D}$, where $\mc{A},~ \mc{B} \in \mathbb{C}^{\textbf{I}(N) \times \textbf{I}(N) }$ and $\mc{A} \neq \mc{B}$. 
\end{remark}
\begin{example}\label{ex1.2}
Consider the tensor $~\mc{A}=(a_{ijkl})
 \in \mathbb{R}^{\overline{2\times3}\times\overline{2\times3}}$ defined in Example \eqref{ex1.1} and  a tensor $~\mc{B}=(b_{ijkl})
 \in \mathbb{R}^{\overline{2\times3}\times\overline{2\times3}}$ with entries
\begin{eqnarray*}
b_{ij11} =
    \begin{pmatrix}
    0 & 0 & 0 \\
    0 & 0 &  0
    \end{pmatrix},~
b_{ij12} =
    \begin{pmatrix}
     1 & 0 & 0\\
     1 & 0 & 0
    \end{pmatrix},~
b_{ij13} =
    \begin{pmatrix}
     1 & 1 & 0\\
     1 & 1 & 0
\end{pmatrix},\\
b_{ij21} =
    \begin{pmatrix}
     1 & 0 & 0\\
     0 & 0 & 0
    \end{pmatrix}, ~
b_{ij22} =
    \begin{pmatrix}
    1 & 1 & 0 \\
    1 & 0 &  0
    \end{pmatrix},~
b_{ij23} =
    \begin{pmatrix}
     1 & 1 & 1\\
     1 & 1 & 0
    \end{pmatrix}.
\end{eqnarray*}
Then $\mc{A}^D = \mc{B}^D  = \mc{O} = \mc{A}^{D}\n\mc{B}^{D} \in \mathbb{R}^{\overline{2\times3}\times\overline{2\times3}}$ and $(\mc{A}*_N\mc{B})^D =(x_{ijkl})
 \in \mathbb{R}^{\overline{2\times3}\times\overline{2\times3}}$ with entries
\begin{eqnarray*}
x_{ij11} =
    \begin{pmatrix}
    0 & 0 & 0 \\
    0 & 0 &  0
    \end{pmatrix},~
x_{ij12} =
    \begin{pmatrix}
     1 & 2 & 0\\
     -1 & -1 & 0
    \end{pmatrix},~
x_{ij13} =
    \begin{pmatrix}
     0 & 0 & 2\\
     0 & -1 & -1
\end{pmatrix},\\
x_{ij21} =
    \begin{pmatrix}
     0 & -1 & 0\\
     2 & 0 & 0
    \end{pmatrix},~
x_{ij22} =
    \begin{pmatrix}
    0 & -1 & -1 \\
    0 & 2 &  0
    \end{pmatrix},~
x_{ij23} =
    \begin{pmatrix}
     0 & 0 & -1\\
     0 & 0 & 1
    \end{pmatrix}.
\end{eqnarray*}

Hence $(\mc{A}\n\mc{B})^{D} \neq \mc{A}^{D}\n\mc{B}^{D}$.   
 \end{example}

The above remark (\ref{reverse}) can be stated as reverse order law for the Drazin inverse of tensors, which is a fundamental in the theory of generalized inverses of tensors. Recently, there has been increasing interest in studying reverse order law of tensors based on different generalized inverses \cite{behera18, Mispa18,  panigrahy}.  In this regard, we discuss one sufficient condition of the reverse order law for the Drazin inverse of tensors, as follows. 

\begin{theorem}\label{mr1.5}
Suppose $\mc{A}, \mc{B}\in \mathbb{C}^{\textbf{I}(N) \times \textbf{I}(N) }$ are each of index $k$. If $\mc{A}\n\mc{B} = \mc{B}\n\mc{A}$, then the followings are true
\begin{enumerate}
    \item[(a)] $\mc{A}^{D}\n\mc{B} = \mc{B}\n\mc{A}^{D}$ and $\mc{A}\n\mc{B}^{D} = \mc{B}^{D}\n\mc{A},$
    \item[(b)] $(\mc{A}\n\mc{B})^{D} = \mc{B}^{D}\n\mc{A}^{D} = \mc{A}^{D}\n\mc{B}^{D}$.
\end{enumerate}
\end{theorem}

\begin{proof}
Since 
\begin{eqnarray*}
\mc{A}^{D}\n\mc{B} &=& \mc{A}^{D}\n\mc{A}\n\mc{A}^{D}\n\mc{B}
= (\mc{A}^{D})^{2}\n\mc{B}\n\mc{A}
= \mc{A}^{D}\n\mc{A}^{D}\n\mc{A}\n\mc{A}^{D}\n\mc{B}\n\mc{A}\\
&=& (\mc{A}^{D})^{3}\n\mc{B}\n\mc{A}^{2}= \cdots ~~= (\mc{A}^{D})^{k+1}\n\mc{B}\n\mc{A}^{k}=(\mc{A}^{D})^{k+1}\n\mc{B}\n\mc{A}^{k+1}\n\mc{A}^{D}\\\
&=&(\mc{A}^{D})^{k+1}\n\mc{A}^{k+1}\n\mc{B}\n\mc{A}^{D}
= \mc{A}^{D}\n\mc{A}\n\mc{B}\n\mc{A}^{D},
\end{eqnarray*}
and 
\begin{eqnarray*}
\mc{B}\n\mc{A}^{D} &=& \mc{B}\n\mc{A}^{D}\n\mc{A}\n\mc{A}^{D}
 = \mc{A}\n\mc{B}\n(\mc{A}^{D})^{2}=\cdots=
 \mc{A}^{k}\n\mc{B}\n(\mc{A}^{D})^{k+1}\\
& =& \mc{A}^{k+1}\n\mc{A}^{D}\n\mc{B}\n(\mc{A}^{D})^{k+1}= \mc{A}^{D}\n\mc{A}^{k+1}\n\mc{B}\n(\mc{A}^{D})^{k+1}\\
 &= &\mc{A}^{D}\n\mc{B}\n\mc{A}^{k+1}\n(\mc{A}^{D})^{k+1}= \mc{A}^{D}\n\mc{B}\n\mc{A}\n\mc{A}^{D}
 = \mc{A}^{D}\n\mc{A}\n\mc{B}\n\mc{A}^{D}. 
\end{eqnarray*}
So we obtain  $\mc{A}^{D}\n\mc{B} = \mc{B}\n\mc{A}^{D}$. Similarly, we can show  $\mc{A}\n\mc{B}^{D} = \mc{B}^{D}\n\mc{A}$. Next we will claim part $(b).$ By using the first part, we get
\begin{eqnarray*}
\mc{B}^{D}\n\mc{A}^{D} &=& \mc{B}^{D}\n\mc{A}^{D}\n\mc{A}\n\mc{A}^{D}
= \mc{A}\n\mc{B}^{D}\n(\mc{A}^{D})^{2}=\cdots= \mc{A}^{k}\n\mc{B}^{D}\n(\mc{A}^{D})^{k+1}\\
&=& \mc{A}^{k+1}\n\mc{A}^{D}\n\mc{B}^{D}\n(\mc{A}^{D})^{k+1}= \mc{A}^{D}\n\mc{A}^{k+1}\n\mc{B}^{D}\n(\mc{A}^{D})^{k+1}\\
&=& \mc{A}^{D}\n\mc{B}^{D}\n\mc{A}^{k+1}\n(\mc{A}^{D})^{k+1}=\mc{A}^{D}\n\mc{B}^{D}\n\mc{A}\n\mc{A}^{D},
\end{eqnarray*}
and 
\begin{eqnarray*}
\mc{A}^{D}\n\mc{B}^{D} &=& \mc{A}^{D}\n\mc{A}\n\mc{A}^{D}\n\mc{B}^{D}
= (\mc{A}^{D})^{2}\n\mc{B}^{D}\n\mc{A}=\cdots= (\mc{A}^{D})^{k+1}\n\mc{B}^{D}\n\mc{A}^{k}\\
&=& (\mc{A}^{D})^{k+1}\n\mc{B}^{D}\n\mc{A}^{k+1}\n\mc{A}^{D}= (\mc{A}^{D})^{k+1}\n\mc{A}^{k}\n\mc{B}^{D}\n\mc{A}\n\mc{A}^{D}\\
&=& \mc{A}^{D}\n\mc{B}^{D}\n\mc{A}\n\mc{A}^{D}.
\end{eqnarray*}
Thus $\mc{B}^{D}\n\mc{A}^{D}=\mc{A}^{D}\n\mc{B}^{D}.$ Now let  $\mc{X} =\mc{B}^{D}\n\mc{A}^{D}$.  Since 
\begin{eqnarray*}
 (\mc{A}\n\mc{B})^{k+1}\n\mc{X}& =& \mc{A}^{k+1}\n\mc{B}^{k+1}\n\mc{B}^{D}\n\mc{A}^{D} = \mc{A}^{k+1}\n\mc{B}^{k}\n\mc{A}^{D} = \mc{B}^{k}\n\mc{A}^{k+1}\n\mc{A}^{D}\\
& = &\mc{B}^{k}\n\mc{A}^{k}
 = (\mc{A}\n\mc{B})^{k} , 
\end{eqnarray*}
\begin{eqnarray*}
  \mc{X}\n(\mc{A}\n\mc{B})\n\mc{X}& =& \mc{B}^{D}\n\mc{A}^{D}\n\mc{A}\n\mc{B}\n\mc{B}^{D}\n\mc{A}^{D} 
= \mc{B}^{D}\n\mc{A}^{D}\n\mc{A}\n\mc{B}\n\mc{A}^{D}\n\mc{B}^{D}\\ &=&\mc{B}^{D}\n\mc{A}^{D}\n\mc{A}\n\mc{A}^{D}\n\mc{B}\n\mc{B}^{D}=\mc{B}^{D}\n\mc{A}^{D}\n\mc{B}\n\mc{B}^{D}\\
&=& \mc{A}^{D}\n\mc{B}^{D}\n\mc{B}\n\mc{B}^{D}
= \mc{A}^{D}\n\mc{B}^{D} = \mc{B}^{D}\n\mc{A}^{D} = \mc{X},  \mbox{ and}
\end{eqnarray*}
\begin{eqnarray*}
 \mc{A}\n\mc{B}\n\mc{X} &=& \mc{A}\n\mc{B}\n\mc{B}^{D}\n\mc{A}^{D} = \mc{A}\n\mc{B}^{D}\n\mc{B}\n\mc{A}^{D} = \mc{B}^{D}\n\mc{A}\n\mc{A}^{D}\n\mc{B}\\
 &=& \mc{B}^{D}\n\mc{A}^{D}\n\mc{A}\n\mc{B} = \mc{X}\n\mc{A}\n\mc{B} .
\end{eqnarray*}
 Thus $\mc{X}$ is the Drazin inverse of $\mc{A}\n\mc{B}.$ Hence $(\mc{A}\n\mc{B})^{D} = \mc{B}^{D}\n\mc{A}^{D} = \mc{A}^{D}\n\mc{B}^{D}$.
\end{proof}

However, the converse of the above theorem need not be true, as shown below with an example.

\begin{example}\label{ex1.10}
Consider the tensor $~\mc{A}=(a_{ijkl})
 \in \mathbb{R}^{\overline{2\times3}\times\overline{2\times3}}$ defined in Example \eqref{ex1.1} and  a tensor $~\mc{B}=(b_{ijkl})
 \in \mathbb{R}^{\overline{2\times3}\times\overline{2\times3}}$ with entries
\begin{eqnarray*}
b_{ij11} =
    \begin{pmatrix}
    0 & 0 & 0 \\
    2 & 0 &  0
    \end{pmatrix},~
b_{ij12} =
    \begin{pmatrix}
     0 & 0 & 0\\
     0 & 4 & 0
    \end{pmatrix},~
b_{ij13} =
    \begin{pmatrix}
     0 & 0 & 0\\
     0 & 0 & 1
\end{pmatrix},\\
b_{ij21} =
    \begin{pmatrix}
     0 & 3 & 0\\
     0 & 0 & 0
    \end{pmatrix},~
b_{ij22} =
    \begin{pmatrix}
    0 & 0 & 5 \\
    0 & 0 &  1
    \end{pmatrix},~
b_{ij23} =
    \begin{pmatrix}
     0 & 0 & 0\\
     0 & 0 & 2
    \end{pmatrix}.
\end{eqnarray*}

Then $ \mc{A}\n\mc{B} =(x_{ijkl})
 \in \mathbb{R}^{\overline{2\times3}\times\overline{2\times3}}$ and $ \mc{B}\n\mc{A} =(y_{ijkl})
 \in \mathbb{R}^{\overline{2\times3}\times\overline{2\times3}},$ where 
\begin{eqnarray*}
x_{ij11} =
    \begin{pmatrix}
    0 & 2 & 2 \\
    0 & 2 &  2
    \end{pmatrix},~
x_{ij12} =
    \begin{pmatrix}
     0 & 0 & 4\\
     0 & 0 & 4
    \end{pmatrix},~
x_{ij13} =
    \begin{pmatrix}
     0 & 0 & 0\\
     0 & 0 & 0
\end{pmatrix},\\
x_{ij21} =
    \begin{pmatrix}
     0 & 0 & 3\\
     0 & 3 & 3
    \end{pmatrix},~
x_{ij22} =
    \begin{pmatrix}
    0 & 0 & 0 \\
    0 & 0 &  5
    \end{pmatrix},~
x_{ij23} =
    \begin{pmatrix}
     0 & 0 & 0\\
     0 & 0 & 0
    \end{pmatrix},
\end{eqnarray*}
and 
\begin{eqnarray*}
y_{ij11} =
    \begin{pmatrix}
    0 & 3 & 5 \\
    0 & 4 &  4
   \end{pmatrix},~
y_{ij12} =
    \begin{pmatrix}
     0 & 0 & 5\\
     0 & 0 & 4
    \end{pmatrix},~
y_{ij13} =
    \begin{pmatrix}
     0 & 0 & 0\\
     0 & 0 & 2
\end{pmatrix},\\
y_{ij21} =
    \begin{pmatrix}
     0 & 0 & 5\\
     0 & 4 & 4
    \end{pmatrix},~
y_{ij22} =
    \begin{pmatrix}
    0 & 0 & 0 \\
    0 & 0 &  3
    \end{pmatrix},~
y_{ij23} =
    \begin{pmatrix}
     0 & 0 & 0\\
     0 & 0 & 0
    \end{pmatrix}.
\end{eqnarray*}
We can easily see that $\mc{A}\n\mc{B} \neq  \mc{B}\n\mc{A}$ but  $(\mc{A}*_N\mc{B})^D = \mc{A}^D*_N\mc{B}^D  = \mc{O}
 \in \mathbb{R}^{\overline{2\times3}\times\overline{2\times3}}$.
\end{example}

The next result presents a characterization of the Drazin inverse of product of two tensors. 
\begin{theorem}\label{mr1.4}
Let $\mc{A}, \mc{B}\in \mathbb{C}^{\textbf{I}(N) \times \textbf{I}(N) }.$  Then 
$(\mc{A}\n\mc{B})^{D} = \mc{A}\n[(\mc{B}\n\mc{A})^{2}]^{D}\n\mc{B}.$
\end{theorem}

\begin{proof}
Let $\mc{X} = \mc{A}\n[(\mc{B}\n\mc{A})^{2}]^{D}\n\mc{B}$ and $k = \max \{\mbox{ind}(\mc{A}\n\mc{B}),~ \mbox{ind}(\mc{B}\n\mc{A})\}$.\\ Now, we have
\begin{eqnarray*}
   (\mc{A}\n\mc{B})^{k+2}\n\mc{X}& =& (\mc{A}\n\mc{B})^{k+1}\n\mc{A}\n\mc{B}\n\mc{A}\n[(\mc{B}\n\mc{A})^{2}]^{D}\n\mc{B}\\
   &=& (\mc{A}\n\mc{B})^{k+1}\n\mc{A}\n(\mc{B}\n\mc{A})\n(\mc{B}\n\mc{A})^{D}\n(\mc{B}\n\mc{A})^{D}\n\mc{B}\\
   &=& (\mc{A}\n\mc{B})^{k+1}\n\mc{A}\n(\mc{B}\n\mc{A})^{D}\n\mc{B}= \mc{A}\n(\mc{B}\n\mc{A})^{k+1}\n(\mc{B}\n\mc{A})^{D}\n\mc{B}\\
   &=& \mc{A}\n(\mc{B}\n\mc{A})^{k}\n\mc{B} = (\mc{A}\n\mc{B})^{k+1},
\end{eqnarray*}
and 
\begin{eqnarray*}
 \mc{X}\n(\mc{A}\n\mc{B})\n\mc{X}& =& \mc{A}\n[(\mc{B}\n\mc{A})^{2}]^{D}\n\mc{B}\n\mc{A}\n\mc{B}\n\mc{A}\n[(\mc{B}\n\mc{A})^{2}]^{D}\n\mc{B}\\
 &=& \mc{A}\n[(\mc{B}\n\mc{A})^{2}]^{D}\n(\mc{B}\n\mc{A})^{2}\n[(\mc{B}\n\mc{A})^{2}]^{D}\n\mc{B}\\
 &=& \mc{A}\n[(\mc{B}\n\mc{A})^{2}]^{D}\n\mc{B} = \mc{X}, 
\end{eqnarray*}
Further, 
\begin{eqnarray*}
\mc{X}\n\mc{A}\n\mc{B} &=& \mc{A}\n[(\mc{B}\n\mc{A})^{D}]^{2}\n\mc{B}\n\mc{A}\n\mc{B}\\
&=& \mc{A}\n(\mc{B}\n\mc{A})^{D}\n(\mc{B}\n\mc{A})*_N(\mc{B}\n\mc{A})^{D}\n\mc{B}\\
&=&\mc{A}\n(\mc{B}\n\mc{A})\n(\mc{B}\n\mc{A})^{D}\n(\mc{B}\n\mc{A})^{D}\n\mc{B}\\
&=& (\mc{A}\n\mc{B})\n\mc{A}\n[(\mc{B}\n\mc{A})^{2}]^{D}\n\mc{B} 
= \mc{A}\n\mc{B}\n\mc{X}.
\end{eqnarray*}
Therefore, $(\mc{A}\n\mc{B})^{D} = \mc{A}\n[(\mc{B}\n\mc{A})^{2}]^{D}\n\mc{B}$.
\end{proof}

\begin{remark}
In general, the Drazin inverse of addition (resp. subtraction) of two different tensor is not equal to individuals, i.e., $(\mc{A}\pm \mc{B})^D \neq \mc{A}^D \pm \mc{B}^D$, as shown below with an example. 
\end{remark}

\begin{example}\label{ex1.4}
Consider the tensor $~\mc{A}=(a_{ijkl})
 \in \mathbb{R}^{\overline{2\times3}\times\overline{2\times3}}$ defined in Example \eqref{ex1.1} and the tensor $~\mc{B}=(a_{ijkl})
 \in \mathbb{R}^{\overline{2\times3}\times\overline{2\times3}}$ defined in Example \eqref{ex1.2}.\\
Then   $\mc{A}^D = \mc{B}^D = \mc{O} = \mc{A}^D + \mc{B}^D \neq (\mc{A}+\mc{B})^D,$ 
where $ \mc{A}^D + \mc{B}^D = (x_{ijkl})
 \in \mathbb{R}^{\overline{2\times3}\times\overline{2\times3}}$ with entries
\begin{eqnarray*}
x_{ij11} =
    \begin{pmatrix}
    0 & 2 & 2 \\
    0 & 2 &  2
    \end{pmatrix},~
x_{ij12} =
    \begin{pmatrix}
     0 & 0 & 4\\
     0 & 0 & 4
    \end{pmatrix},~
x_{ij13} =
    \begin{pmatrix}
     0 & 0 & 0\\
     0 & 0 & 0
\end{pmatrix},\\
x_{ij21} =
    \begin{pmatrix}
     0 & 0 & 3\\
     0 & 3 & 3
    \end{pmatrix},~
x_{ij22} =
    \begin{pmatrix}
    0 & 0 & 0 \\
    0 & 0 &  5
    \end{pmatrix},~
x_{ij23} =
    \begin{pmatrix}
     0 & 0 & 0\\
     0 & 0 & 0
    \end{pmatrix}.
\end{eqnarray*}
Similarly, one can show  $ \mc{A}^D - \mc{B}^D  \neq (\mc{A}-\mc{B})^D.$
\end{example}
At this stage one may be excited to know when does the Drazin  inverse  of  addition  (resp.  subtraction)  of  two  different tensor will be equal  to the individuals. The following theorem answer to this question. 
\begin{theorem}\label{mr1.14}
Let $\mc{A}, \mc{B} \in \mathbb{C}^{\textbf{I}(N) \times \textbf{I}(N) }$ be of index $k$. If $\mc{A}\n\mc{B} = \mc{B}\n\mc{A} = \mc{O}$, then\
\begin{enumerate}
    \item[(a)] $(\mc{A}+\mc{B})^{D} = \mc{A}^{D}+\mc{B}^{D}$,
    \item[(b)] $(\mc{A}-\mc{B})^{D} = \mc{A}^{D}-\mc{B}^{D}$.
\end{enumerate}
\end{theorem}

\begin{proof}
Since $\mc{A}\n\mc{B} = \mc{B}\n\mc{A} = \mc{O}$, we have $(\mc{A}+\mc{B})^{k+1} = \mc{A}^{k+1}+\mc{B}^{k+1}$ and 
\begin{equation}
\mc{A}^{k+1}\n\mc{B}^{D} = \mc{A}^{k}\n\mc{A}\n\mc{B}^{D}\n\mc{B}\n\mc{B}^{D} = \mc{A}^{k}\n\mc{A}\n\mc{B}\n\mc{B}^{D}\n\mc{B}^{D} = \mc{O}.   
\end{equation}
 Similarly, we can show  $\mc{B}^{k+1}\n\mc{A}^{D} = \mc{O}$. Now using these relations and the Theorem \ref{mr1.5}, we obtain
 \begin{eqnarray}\label{add1}
(\mc{A}+\mc{B})^{k+1}\n(\mc{A}^{D}+\mc{B}^{D}) 
&=& \mc{A}^{k}+\mc{B}^{k}= (\mc{A}+\mc{B})^{k},
 \end{eqnarray}
 \begin{eqnarray}\label{add2}
 (\mc{A}^{D}+\mc{B}^{D})\n(\mc{A}+\mc{B})\n(\mc{A}^{D}+\mc{B}^{D})
 &=& \mc{A}^{D}+\mc{B}^{D},
 \end{eqnarray}
and
\begin{eqnarray}\label{add3}
(\mc{A}+\mc{B})\n(\mc{A}^{D}+\mc{B}^{D}) 
&=& (\mc{A}^{D}+\mc{B}^{D})\n(\mc{A}+\mc{B}).
\end{eqnarray}
From \eqref{add1}, \eqref{add2} and \eqref{add3}, we conclude $\mc{A}^{D}+\mc{B}^{D}$ is the Drazin inverse of $\mc{A}+\mc{B}$. Hence proved part $(a).$ Using the similar lines, we can show  $(\mc{A}-\mc{B})^D=\mc{A}^{D}-\mc{B}^{D}.$ 
\end{proof}
In case of the group inverse, we can relax one sufficient condition of Theorem \ref{mr1.14} and computes the group inverse of the tensor $\mc{A}+\mc{B}$ as per the following theorem.
\begin{theorem}\label{mg1.1}
Suppose $\mc{A}, \mc{B} \in \mathbb{C}^{\textbf{I}(N) \times \textbf{I}(N) }$ are of index one. If $\mc{A}\n\mc{B} = \mc{O}$, then
\begin{equation*}
(\mc{A}+\mc{B})^{\#} = (I-\mc{B}\n\mc{B}^{\#})\n\mc{A}^{\#}+\mc{B}^{\#}\n(I-\mc{A}\n\mc{A}^{\#}).
\end{equation*}
\end{theorem}

\begin{proof}
Let $\mc{X} = (I-\mc{B}\n\mc{B}^{\#})\n\mc{A}^{\#}+\mc{B}^{\#}\n(I-\mc{A}\n\mc{A}^{\#})$. By using $\mc{A}\n\mc{B} = \mc{O}$, we obtain 
\begin{eqnarray*}
 \mc{A}\n\mc{X}\n\mc{A} &=& \mc{A}\n\mc{A}^{\#}\n\mc{A}-\mc{A}\n\mc{B}\n\mc{B}^{\#}\n\mc{A}^{\#}\n\mc{A}+\mc{A}\n\mc{B}^{\#}\n\mc{A}-\mc{A}\n\mc{B}^{\#}\n\mc{A}\n\mc{A}^{\#}\n\mc{A}\\
 &=& \mc{A}+\mc{A}\n\mc{B}^{\#}\n\mc{A}-\mc{A}\n\mc{B}^{\#}\n\mc{A}=\mc{A},
\end{eqnarray*}
and 
\begin{eqnarray*}
 \mc{A}\n\mc{X}\n\mc{B} &=& \mc{A}\n\mc{A}^{\#}\n\mc{B}-\mc{A}\n\mc{B}\n\mc{B}^{\#}\n\mc{A}^{\#}\n\mc{B}+\mc{A}\n\mc{B}^{\#}\n\mc{B}-\mc{A}\n\mc{B}^{\#}\n\mc{A}\n\mc{A}^{\#}\n\mc{B}\\
 &=& \mc{A}^{\#}\n\mc{A}\n\mc{B}+\mc{A}\n\mc{B}\n\mc{B}^{\#}=\mc{O}.
\end{eqnarray*}
Similarly, we can show $\mc{B}\n\mc{X}\n\mc{B} = \mc{B}$ and $\mc{B}\n\mc{X}\n\mc{A} = \mc{O}$. Also we have 
\begin{eqnarray*}
 \mc{X}\n\mc{A}\n\mc{X} &=& (\mc{A}^{\#}\n\mc{A}-\mc{B}\n\mc{B}^{\#}\n\mc{A}^{\#}\n\mc{A}+\mc{B}^{\#}\n\mc{A}-\mc{B}^{\#}\n\mc{A}\n\mc{A}^{\#}\n\mc{A})\n\\
 & & (\mc{A}^{\#}-\mc{B}\n\mc{B}^{\#}\n\mc{A}^{\#}+\mc{B}^{\#}-\mc{B}^{\#}\n\mc{A}\n\mc{A}^{\#})\\
 &=& (\mc{A}^{\#}\n\mc{A}-\mc{B}\n\mc{B}^{\#}\n\mc{A}^{\#}\n\mc{A})\n(\mc{A}^{\#}-\mc{B}\n\mc{B}^{\#}\n\mc{A}^{\#}+\mc{B}^{\#}-\mc{B}^{\#}\n\mc{A}\n\mc{A}^{\#})\\
 &=& \mc{A}^{\#}-\mc{B}\n\mc{B}^{\#}\n\mc{A}^{\#},
\end{eqnarray*}
and $\mc{X}\n\mc{B}\n\mc{X} = \mc{B}^{\#}-\mc{B}^{\#}\n\mc{A}\n\mc{A}^{\#}$. Now using the above results, we get  
\begin{eqnarray*}
(\mc{A}+\mc{B})\n\mc{X}\n(\mc{A}+\mc{B}) = \mc{A}\n\mc{X}\n\mc{A}+\mc{A}\n\mc{X}\n\mc{B}+\mc{B}\n\mc{X}\n\mc{A}+\mc{B}\n\mc{X}\n\mc{B} = \mc{A}+\mc{B}, 
\end{eqnarray*}
and 
\begin{eqnarray*}
\mc{X}\n(\mc{A}+\mc{B})\n\mc{X} = \mc{X}\n\mc{A}\n\mc{X}+\mc{X}\n\mc{B}\n\mc{X} = \mc{A}^{\#}-\mc{B}\n\mc{B}^{\#}\n\mc{A}^{\#}+\mc{B}^{\#}-\mc{B}^{\#}\n\mc{A}\n\mc{A}^{\#} = \mc{X},
\end{eqnarray*}
further, 
\begin{eqnarray*}
(\mc{A}+\mc{B})\n\mc{X} &=& (\mc{A}\n\mc{A}^{\#}-\mc{A}\n\mc{B}\n\mc{B}^{\#}\n\mc{A}^{\#}+\mc{A}\n\mc{B}^{\#}-\mc{A}\n\mc{B}^{\#}\n\mc{A}\n\mc{A}^{\#})\\
& &+(\mc{B}\n\mc{A}^{\#}-\mc{B}\n\mc{B}\n\mc{B}^{\#}\n\mc{A}^{\#}+\mc{B}\n\mc{B}^{\#}-\mc{B}\n\mc{B}^{\#}\n\mc{A}\n\mc{A}^{\#})\\
&=& \mc{A}\n\mc{A}^{\#}+\mc{B}\n\mc{B}^{\#}-\mc{B}\n\mc{B}^{\#}\n\mc{A}\n\mc{A}^{\#}\\
&=& (\mc{A}^{\#}\n\mc{A}-\mc{B}\n\mc{B}^{\#}\n\mc{A}^{\#}\n\mc{A}+\mc{B}^{\#}\n\mc{A}-\mc{B}^{\#}\n\mc{A}\n\mc{A}^{\#}\n\mc{A})\\
& &+(\mc{A}^{\#}\n\mc{B}-\mc{B}\n\mc{B}^{\#}\n\mc{A}^{\#}\n\mc{B}+\mc{B}^{\#}\n\mc{B}-\mc{B}^{\#}\n\mc{A}\n\mc{A}^{\#}\n\mc{B})\\
&=& \mc{X}\n\mc{A}+\mc{X}\n\mc{B}= \mc{X}\n(\mc{A}+\mc{B}).
\end{eqnarray*}
Therefore, $(\mc{A}+\mc{B})^{\#} = (I-\mc{B}\n\mc{B}^{\#})\n\mc{A}^{\#}+\mc{B}^{\#}\n(I-\mc{A}\n\mc{A}^{\#})$.
\end{proof}
\subsection{Computation of the Drazin inverse of tensors}
Now we present the main objective of this paper, which yields the  construction of the Drazin inverse of a tensor using other generalized inverses and  decomposition of tensors. One can find the matrix version of these results in \cite{bar,ben}. 

\begin{theorem}\label{mr1.17}
Let $\mc{A} \in \mathbb{C}^{\textbf{I}(N) \times \textbf{I}(N) }$  be a tensor with ind$(\mc{A})=k.$ Then 
\begin{equation*}
\mc{A}^{D} = \mc{A}^{k}\n(\mc{A}^{2k+1})^{\dag}\n\mc{A}^{k}
\end{equation*}
\end{theorem}
\begin{proof}
By using the definition of the Drazin inverse, we obtain
\begin{equation*}
\mc{A}^{k} = \mc{A}^{k+1}\n\mc{A}^{D} = \mc{A}^{k+2}\n(\mc{A}^{D})^{2} =\cdots = \mc{A}^{2k}\n(\mc{A}^{D})^{k} = \mc{A}^{2k+1}\n(\mc{A}^{D})^{k+1}.   
\end{equation*}
Let $\mc{X}=\mc{A}^{k}\n(\mc{A}^{2k+1})^{\dag}\n\mc{A}^{k}.$ Now 
\begin{eqnarray*}
  \mc{A}^{k+1}\n{X} &=& \mc{A}^{k+1}\n\mc{A}^{k}\n(\mc{A}^{2k+1})^{\dag}\n\mc{A}^{k} = \mc{A}^{2k+1}\n(\mc{A}^{2k+1})^{\dag}\n\mc{A}^{2k+1}\n(\mc{A}^{D})^{k+1}\\
  &=& \mc{A}^{2k+1}\n(\mc{A}^{D})^{k+1}
  = \mc{A}^{k+1}\n\mc{A}^{k}\n(\mc{A}^{D})^{k}\n\mc{A}^{D} = \mc{A}^{k+1}\n\mc{A}\n\mc{A}^{D}\n\mc{A}^{D}\\
  &=& \mc{A}^{k+1}\n\mc{A}^{D}\n\mc{A}\n\mc{A}^{D}
  = \mc{A}^{k+1}\n\mc{A}^{D} = \mc{A}^{k}, 
\end{eqnarray*}
\begin{eqnarray*}
   \mc{X}\n\mc{A}\n\mc{X}& =& \mc{A}^{k}\n(\mc{A}^{2k+1})^{\dag}\n\mc{A}^{k}\n\mc{A}\n\mc{A}^{k}\n(\mc{A}^{2k+1})^{\dag}\n\mc{A}^{k} \\
&=& \mc{A}^{k}\n(\mc{A}^{2k+1})^{\dag}\n\mc{A}^{2k+1}\n(\mc{A}^{2k+1})^{\dag}\n\mc{A}^{2k+1}\n(\mc{A}^{D})^{k+1}\\
&=& \mc{A}^{k}\n(\mc{A}^{2k+1})^{\dag}\n\mc{A}^{2k+1}\n(\mc{A}^{D})^{k+1} =\mc{A}^{k}\n(\mc{A}^{2k+1})^{\dag}\n\mc{A}^{k} = \mc{X},
\end{eqnarray*}
and 
\begin{eqnarray*}
 \mc{A}\n\mc{X} &=& \mc{A}\n\mc{A}^{k}\n(\mc{A}^{2k+1})^{\dag}\n\mc{A}^{k}= \mc{A}\n\mc{A}^{2k}\n(\mc{A}^{D})^{k}\n(\mc{A}^{2k+1})^{\dag}\n\mc{A}^{2k+1}\n(\mc{A}^{D})^{k+1}\\
 &=& (\mc{A}^{D})^{k}\n\mc{A}^{2k+1}\n(\mc{A}^{2k+1})^{\dag}\n\mc{A}^{2k+1}\n(\mc{A}^{D})^{k+1} 
= (\mc{A}^{D})^{k}\n\mc{A}^{2k+1}\n(\mc{A}^{D})^{k+1}\\
&=& (\mc{A}^{D})^{k+1}\n\mc{A}^{2k+1}\n(\mc{A}^{D})^{k}
= (\mc{A}^{D})^{k+1}\n\mc{A}^{2k+1}\n(\mc{A}^{2k+1})^{\dag}\n\mc{A}^{2k+1}\n(\mc{A}^{D})^{k}\\
&=& \mc{A}^{2k+1}\n(\mc{A}^{D})^{k+1}\n(\mc{A}^{2k+1})^{\dag}\n\mc{A}\n\mc{A}^{2k}\n(\mc{A}^{D})^{k} \\
&=& \mc{A}^{k}\n(\mc{A}^{2k+1})^{\dag}\n\mc{A}^{k+1} = \mc{X}\n\mc{A}.  
\end{eqnarray*}
Therefore, by Definition \ref{dz1.1}, we obtain  $\mc{A}^{D} =\mc{X}= \mc{A}^{k}\n(\mc{A}^{2k+1})^{\dag}\n\mc{A}^{k}.$
\end{proof}
Further, the Drazin inverse of a tensor is constructed within the framework of Moore-Penrose inverse, as follows.

\begin{theorem}\label{mr1.16}
Let $\mc{A} \in \mathbb{C}^{\textbf{I}(N) \times \textbf{I}(N) }$ be a tensor with  index $k.$ If $l$ is any integer with $l\geq k,$ then  $\mc{A}^{D} = \mc{X}^{\dag},$ where $\mc{X} = (\mc{A}^{l})^{\dag}\n\mc{A}^{2l+1}\n(\mc{A}^{l})^{\dag}.$
\end{theorem}

\begin{proof}
To claim the theorem, it is enough to show  $(\mc{A}^{D})^{\dag} = \mc{X}$. As $\mc{A}^{D} = \mc{A}^{D}\n\mc{A}\n\mc{A}^{D} = (\mc{A}^{D})^{2}\n\mc{A} = \mc{A}\n(\mc{A}^{D})^{2}$. Utilizing the fact $\mc{A}^{D} = \mc{A}^{D}\n\mc{A}\n\mc{A}^{D}$, $l$-times repetitively, we obtain
\begin{equation}\label{eq3.3}
  \mc{A}^{D} = (\mc{A}^{D})^{l+1}\n\mc{A}^{l} = \mc{A}^{l}\n(\mc{A}^{D})^{l+1}.  
\end{equation}
Now by using Eq. (\ref{eq3.3}), we have
\begin{eqnarray*}
 \mc{A}^{D}\n\mc{X}\n\mc{A}^{D} &=& \mc{A}^{D}\n(\mc{A}^{l})^{\dag}\n\mc{A}^{2l+1}\n(\mc{A}^{l})^{\dag}\n\mc{A}^{D}\\
 &=& (\mc{A}^{D})^{l+1}\n\mc{A}^{l}\n(\mc{A}^{l})^{\dag}\n\mc{A}^{l}\n\mc{A}\n\mc{A}^{l}\n(\mc{A}^{l})^{\dag}\n\mc{A}^{l}\n(\mc{A}^{D})^{l+1}\\
 &=& (\mc{A}^{D})^{l+1}\n\mc{A}^{l}\n\mc{A}\n\mc{A}^{l}\n(\mc{A}^{D})^{l+1}= \mc{A}^{D}\n\mc{A}\n\mc{A}^{D}= \mc{A}^{D},
\end{eqnarray*}
\begin{eqnarray*}
\mc{X}\n\mc{A}^{D}\n\mc{X} &=& (\mc{A}^{l})^{\dag}\n\mc{A}^{2l+1}\n(\mc{A}^{l})^{\dag}\n\mc{A}^{D}\n(\mc{A}^{l})^{\dag}\n\mc{A}^{2l+1}\n(\mc{A}^{l})^{\dag}\\
 &=& (\mc{A}^{l})^{\dag}\n\mc{A}^{l+1}\n\mc{A}^{l}\n(\mc{A}^{l})^{\dag}\n\mc{A}^{l}\n(\mc{A}^{D})^{l+1}\n(\mc{A}^{l})^{\dag}\n\mc{A}^{2l+1}\n(\mc{A}^{l})^{\dag}\\
 &=& (\mc{A}^{l})^{\dag}\n\mc{A}^{l+1}\n\mc{A}^{D}\n(\mc{A}^{l})^{\dag}\n\mc{A}^{2l+1}\n(\mc{A}^{l})^{\dag}\\
 &=& (\mc{A}^{l})^{\dag}\n\mc{A}\n\mc{A}^{D}\n\mc{A}^{l}\n(\mc{A}^{l})^{\dag}\n\mc{A}^{l}\n\mc{A}^{l+1}\n(\mc{A}^{l})^{\dag}\\ 
 &=& (\mc{A}^{l})^{\dag}\n\mc{A}\n\mc{A}^{D}\n\mc{A}^{l}\n\mc{A}^{l+1}\n(\mc{A}^{l})^{\dag}= \mc{X}.
\end{eqnarray*}
Since 
\begin{eqnarray*}
   \mc{X}\n\mc{A}^{D}& =& (\mc{A}^{l})^{\dag}\n\mc{A}^{2l+1}\n(\mc{A}^{l})^{\dag}\n\mc{A}^{D} = (\mc{A}^{l})^{\dag}\n\mc{A}^{l+1}\n\mc{A}^{l}\n(\mc{A}^{l})^{\dag}\n\mc{A}^{l}\n(\mc{A}^{D})^{l+1}\\
   &=& 
(\mc{A}^{l})^{\dag}\n\mc{A}^{l+1}\n\mc{A}^{l}\n(\mc{A}^{D})^{l+1} = (\mc{A}^{l})^{\dag}\n\mc{A}^{l+1}\n\mc{A}^{D} = (\mc{A}^{l})^{\dag}\n\mc{A}^{l},
\end{eqnarray*}
and $\mc{A}^{D}\n\mc{X} = \mc{A}^{D}\n(\mc{A}^{l})^{\dag}\n\mc{A}^{2l+1}\n(\mc{A}^{l})^{\dag}=
 (\mc{A}^{D})^{l+1}\n\mc{A}^{l}\n(\mc{A}^{l})^{\dag}\n\mc{A}^{l}\n\mc{A}^{l+1}\n(\mc{A}^{l})^{\dag} = (\mc{A}^{D})^{l+1}\n\mc{A}^{l}\n\mc{A}^{l+1}\n(\mc{A}^{l})^{\dag}= \mc{A}^{D}\n\mc{A}^{l+1}\n(\mc{A}^{l})^{\dag} = \mc{A}^{l}\n(\mc{A}^{l})^{\dag}.$ Therefore,    $(\mc{X}\n\mc{A}^{D})^{*} = ((\mc{A}^{l})^{\dag}\n\mc{A}^{l})^{*} = (\mc{A}^{l})^{\dag}\n\mc{A}^{l} = \mc{X}\n\mc{A}^{D},$ and  $(\mc{A}^{D}\n\mc{X})^{*} = (\mc{A}^{l}\n(\mc{A}^{l})^{\dag})^{*} = \mc{A}^{l}\n(\mc{A}^{l})^{\dag} = \mc{A}^{D}\n\mc{X}.$ Hence $(\mc{A}^{D})^{\dag} = \mc{X}$. 
\end{proof}

Using method of induction and definition of the Drazin inverse of the tensor, we present another characterization for computation of the Drazin inverse of tensors, as follows.  
\begin{preposition}
\label{mrr1.01}
Let $\mc{A} \in \mathbb{C}^{\textbf{I}(N) \times \textbf{I}(N) }$ and ind$(\mc{A})=k.$ If there exists a tensor $\mc{X} \in \mathbb{C}^{\textbf{I}(N) \times \textbf{I}(N) }$ such that $\mc{A}\n\mc{X}^{k+1} = \mc{X}^{k}$ and $\mc{X}\n\mc{A}^{k+1} = \mc{A}^{k}$, then for  $m\in \mathbb{N},$ the followings hold
\begin{enumerate}
    \item[(a)] $\mc{A}^{k} = \mc{X}^{m}\n\mc{A}^{k+m}$,
    \item[(b)] $\mc{X}^{k}\n\mc{A}^{k} = \mc{X}^{k+m}\n\mc{A}^{k+m}$,
    \item[(c)] $\mc{A}^{D} = \mc{X}^{k+1}\n\mc{A}^{k}.$
\end{enumerate}
\end{preposition}
We now discuss the method of construction of the group inverse of a tensor using a $\{1\}$-inverse of $\mc{A}$.
\begin{theorem}\label{mr1.15}
Let $\mc{A} \in \mathbb{C}^{\textbf{I}(N) \times \textbf{I}(N) }$  and ind$(\mc{A})=1.$ Then $\mc{A}^{\#} = \mc{A}\n(\mc{A}^{3})^{(1)}\n\mc{A},$
where $(\mc{A}^{3})^{(1)}$ is an arbitrary $\{1\}$-inverse of $\mc{A}^{3}.$
\end{theorem}

\begin{proof}
Let $\mc{X} = \mc{A}\n(\mc{A}^{3})^{(1)}\n\mc{A}$. Since  
\begin{eqnarray*}
\mc{A}\n\mc{X}\n\mc{A} &=& \mc{A}^2\n(\mc{A}^{3})^{(1)}\n\mc{A}^2= \mc{A}\n\mc{A}\n\mc{A}^{\#}\n\mc{A}\n(\mc{A}^{3})^{(1)}\n\mc{A}\n\mc{A}^{\#}\n\mc{A}\n\mc{A}\\
&=& \mc{A}^{\#}\n\mc{A}^{3}\n(\mc{A}^{3})^{(1)}\n\mc{A}^{3}\n\mc{A}^{\#}
= \mc{A}^{\#}\n\mc{A}^{3}\n\mc{A}^{\#}=\mc{A}, 
\end{eqnarray*}
\begin{eqnarray*}
\mc{X}\n\mc{A}\n\mc{X} &=& \mc{A}\n(\mc{A}^{3})^{(1)}\n\mc{A}^3\n(\mc{A}^{3})^{(1)}\n\mc{A}=(\mc{A}^{\#})^{2}\n\mc{A}^{3}\n(\mc{A}^{3})^{(1)}\n\mc{A}^{3}\n(\mc{A}^{3})^{(1)}\n\mc{A}\\
&=& (\mc{A}^{\#})^{2}\n\mc{A}^{3}\n(\mc{A}^{3})^{(1)}\n\mc{A}=\mc{A}\n(\mc{A}^{3})^{(1)}\n\mc{A}=\mc{X},
\end{eqnarray*}
and 
\begin{eqnarray*}
\mc{A}\n\mc{X} &=& \mc{A}\n\mc{A}\n(\mc{A}^{3})^{(1)}\n\mc{A}=  \mc{A}^{\#}\n\mc{A}^3\n(\mc{A}^{3})^{(1)}\n\mc{A}^3\n\left(\mc{A}^{\#}\right)^2=(\mc{A}^{\#})^{2}\n\mc{A}^{3}\n\mc{A}^{\#}\\
&=& (\mc{A}^{\#})^{2}\n\mc{A}^{3}\n(\mc{A}^{3})^{(1)}\n\mc{A}^{3}\n\mc{A}^{\#}= \mc{A}\n(\mc{A}^{3})^{(1)}\n\mc{A}\n\mc{A}= \mc{X}\n\mc{A}.
\end{eqnarray*}

Thus $\mc{X}$ is the group inverse of $\mc{A}.$ Hence  $\mc{A}^{\#} = \mc{A}\n(\mc{A}^{3})^{(1)}\n\mc{A}.$
\end{proof}
For computation of the Drazin inverse, the index-$1$ nilpotent decomposition plays an important role, which we discuss in the next theorem. 
\begin{theorem}\label{mr1.11}
Let $\mc{A} \in \mathbb{C}^{\textbf{I}(N) \times \textbf{I}(N) }.$  Then $\mc{A}$ has a unique decomposition $\mc{A} = \mc{B}+\mc{N},$
such that $\mc{B}$ is of index $1$, $\mc{N}$ is nilpotent, and $\mc{N}\n\mc{B} = \mc{B}\n\mc{N} = \mc{O}.  $
Further, $  (\mc{A}^{D})^{\#}=\mc{B}.$
\end{theorem}

\begin{proof}
First we prove, if the tensor $\mc{A}$ has a decomposition, $\mc{A} = \mc{B}+\mc{N},$
such that $\mc{B}$ is of index $1$, $\mc{N}$ is nilpotent, and $\mc{N}\n\mc{B} = \mc{B}\n\mc{N} =\mc{O} $ then $\mc{A}^D=\mc{B}^{\#}.$ Subsequently, the uniqueness follows from the uniqueness of the group inverse. Since $\mc{B}^{\#} = \mc{B}\n(\mc{B}^{\#})^{2} = (\mc{B}^{\#})^{2}\n\mc{B},$ so post-multiplying by $\mc{N},$ we obtain $\mc{B}^{\#}\n\mc{N} = \mc{O}$. Similarly, $\mc{N}\n\mc{B}^{\#} =\mc{O}$. Further, since $\mc{N}$ nilpotent tensor, then there exists a positive integer $k$ such that $\mc{N}^k=\mc{O}.$ This yields $\mc{A}^{k} = (\mc{B}+\mc{N})^{k} = \mc{B}^{k}+\mc{N}^{k}=\mc{B}^k.$ Now using these results, we get $ \mc{A}^{k+1}\n\mc{B}^{\#} = \mc{B}^{k+1}\n\mc{B}^{\#} = \mc{B}^{k}=\mc{A}^k,$ $\mc{B}^{\#}\n\mc{A}\n\mc{B}^{\#}=\mc{B}\n(\mc{B}^{\#})^{2} =\mc{B}^{\#},$ and $\mc{A}\n\mc{B}^{\#} =\mc{B}\n\mc{B}^{\#} = \mc{B}^{\#}\n\mc{A}.$ Thus $\mc{B}^{\#}$  is the Drazin inverse of $\mc{A}.$ Hence $\mc{B}=\left(\mc{A}^{D}\right)^{\#}.$ Now consider $\mc{B}=\left(\mc{A}^{D}\right)^{\#},$ and prove the decomposition $\mc{A}=\mc{B}+\mc{N}$ satisfies the required conditions.  
In view of the Theorem \ref{mr1.2} $(b)$ we obtain $\mc{N}\n\mc{B} = \mc{A}\n\mc{A}^2\n\mc{A}^{D}-\mc{A}^2\n\mc{A}^{D}\n\mc{A}^2\n\mc{A}^D=\mc{A}^3\n\mc{A}^D- \mc{A}^{3}\n\mc{A}^{D}=\mc{O}.$ Similarly, one can show $\mc{B}\n\mc{N}=\mc{O}.$ Now to claim $\mc{N}$ is nilpotent, we consider ind $(\mc{A})=k.$ Then 
\begin{equation*}
\mc{A}^{k} = \mc{B}^{k}+\mc{N}^{k} = ((\mc{A}^{D})^{\#})^{k}+\mc{N}^{k} = \mc{A}^{2k}\n(\mc{A}^{D})^{k}+\mc{N}^{k} = \mc{A}^{k}+\mc{N}^{k}.
\end{equation*}
Thus $\mc{N}^{k} = \mc{O}.$ Hence completes the proof.
\end{proof}


Using the definition of reshape operation of a tensor $\mc{A} = \mc{F}\n\mc{H}\n\mc{G} \in \mathbb{C}^{\textbf{I}(M) \times \textbf{J}(N)}$, we obtain  $rshrank(\mc{A})\leq rshrank(\mc{H})$. On the other hand, suppose $\mc{G}_{r}$ is the right inverse of $\mc{G}$ and $\mc{F}_{l}$ is the left inverse of $\mc{F}$. Now pre-multiplying $\mc{F}_{l}$ and post-multiplying $\mc{G}_{r}$ to the tensor $\mc{A}$,  we obtain $\mc{H} = \mc{F}_{l}\m\mc{A}\n\mc{G}_{r}$. This implies $rshrank(\mc{H})\leq rshrank(\mc{A})$. 
As a consequence, this is stated in the next result.
\begin{theorem}\label{mr2.11}
 Let $\mc{A} \in \mathbb{C}^{\textbf{I}(M) \times \textbf{J}(N)}$. Suppose $\mc{A} = \mc{F}\n\mc{H}\n\mc{G}$, where $\mc{F}\in \mathbb{C}^{\textbf{I}(M) \times \textbf{K}(L)}$ and $\mc{G}\in \mathbb{C}^{\textbf{K}(L)\times \textbf{J}(N)}.$ If $rshrank(\mc{F})=\textbf{K}(L)=rshrank(\mc{G}),$ then $rshrank(\mc{A}) = rshrank(\mc{H})$.
\end{theorem}

One of its application for computing the Moore-Penrose inverse is already presented in the recent literature \cite{behera18}. The computation of the group inverse via full rank tensor factorization is discussed in the next theorem.

\begin{theorem}\label{mr1.13}
Suppose $\mc{A} \in \mathbb{C}^{\textbf{I}(N) \times \textbf{I}(N) }$ has the full-rank factorization, $\mc{A} = \mc{F}\n\mc{G},$ where $\mc{F}\in \mathbb{C}^{\textbf{I}(N) \times K_1 \times\cdots\times K_L}$  and $\mc{G}\in \mathbb{C}^{K_1\times\cdots\times K_L \times \textbf{I}(N)}.$  Then $\mc{A}$ is group invertible if and only if $\mc{G}\n\mc{F}$ is nonsingular. Further, we have  $ \mc{A}^{\#} = \mc{F}\n(\mc{G}\n\mc{F})^{-2}\n\mc{G}.
$
\end{theorem}

\begin{proof}
Let $\mc{A} = \mc{F}\n\mc{G}$ and  $rshrank(\mc{A}) = r$. Now using Theorem \ref{mr2.11}, we obtain, $rshrank(\mc{A}^{2}) = rshrank(\mc{G}\n\mc{F})$. Therefore, $rshrank(\mc{A}^{2}) = rshrank(\mc{A})$ {\it if and only if} $\mc{G}\n\mc{F}$ is nonsingular. Hence completes first part of the theorem.  To claim the remaining, let $\mc{X} = \mc{F}\n(\mc{G}\n\mc{F})^{-2}\n\mc{G}$. Now we have 
\begin{eqnarray}\label{FR1}
 \mc{A}\n\mc{X}\n\mc{A}
&=& \mc{F}\n\mc{G}\n\mc{F}\n(\mc{G}\n\mc{F})^{-2}\n\mc{G}\n\mc{F}\n\mc{G} =\mc{F}\n\mc{G} = \mc{A}, \\
\mc{X}\n\mc{A}\n\mc{X} &=& \mc{F}\n(\mc{G}\n\mc{F})^{-2}\n\mc{G}\n\mc{F}\n\mc{G}\n\mc{F}\n(\mc{G}\n\mc{F})^{-2}\n\mc{G} 
= \mc{X}, \label{FR2}
\end{eqnarray}
\begin{equation}\label{FR3}
\mc{A}\n\mc{X} = \mc{F}\n\mc{G}\n\mc{F}\n(\mc{G}\n\mc{F})^{-2}\n\mc{G} = \mc{F}\n(\mc{G}\n\mc{F})^{-2}\n\mc{G}\n\mc{F}\n\mc{G} = \mc{X}\n\mc{A}.
\end{equation}
Thus, from \eqref{FR1}, \eqref{FR2} and \eqref{FR3} one can conclude $\mc{X}$ is the group inverse of $\mc{A}.$ 
\end{proof}

\section{W-weighted Drazin inverse}
The W-weighted Drazin inverse, first introduced by Cline and Greville \cite{cline1980} for rectangular matrices. Further, some characterization has given in \cite{wei2002c}. Also, it extended to linear operators \cite{we2002}. In this section, we introduce the W-weighted Drazin inverse for arbitrary order tensors via Einstein product, which is a generalization of the generalized inverses of rectangular matrices, as follows.
\begin{definition}\label{wdraz}
Let $\mc{B} \in \mathbb{C}^{\textbf{I}(M)\times \textbf{J}(N)}$ and  $\mc{W} \in
\mathbb{C}^{\textbf{J}(N) \times \textbf{I}(M)}.$ If a tensor $\mc{X}\in \mathbb{C}^{\textbf{I}(M)\times \textbf{J}(N)}$ satisfying
\begin{enumerate}
    \item[(a)] $\left(\mc{B}\n\mc{W}\right)^{k+1}\m\mc{X}\n\mc{W}=\left(\mc{B}\n\mc{W}\right)^k$ for some positive integer $k,$
    \item[(b)] $\mc{X}\n\mc{W}\m\mc{B}\n\mc{W}\m\mc{X} = \mc{X},$
    \item[(c)] $\mc{B}\n\mc{W}\m\mc{X} = \mc{X}\n\mc{W}\m\mc{B},$
\end{enumerate}
is called the W-weighted Drazin inverse of $\mc{B}$ and denoted by $\mc{B}^{D,W}.$ 
\end{definition}
Using the Defintion \ref{wdraz} and Theorem \ref{mr1.4} one can prove the following results.

\begin{lemma}\label{lem4.1}
Let  $\mc{B} \in \mathbb{C}^{\textbf{I}(M) \times \textbf{J}(N)}$ and  $\mc{W} \in
\mathbb{C}^{\textbf{J}(N) \times \textbf{I}(M) }$, be any tensors. Then 
\begin{enumerate}
    \item[(a)] $(\mc{B}\n\mc{W})^{D} = \mc{B}\n[(\mc{W}\m\mc{B})^{2}]^{D}\n\mc{W}$,
    \item[(b)] $\mc{B}=\mc{B}^{D,W} \mbox{ if and only if }\mc{B}=\mc{B}\n\mc{W}\m\mc{B}\n\mc{W}\m\mc{B}.$
\end{enumerate}
\end{lemma}

In connection with the above lemma of an arbitrary-order tensor, the following theorem collects some useful identities of the W-weighted Drazin inverse of tensors.

\begin{theorem}\label{w1.2}
Let $\mc{B} \in \mathbb{C}^{\textbf{I}(M)  \times\textbf{J}(N)}$ and $\mc{W} \in
\mathbb{C}^{\textbf{J}(N) \times \textbf{I}(M) }.$ Then for every positive integer $p$, the following holds
\begin{enumerate}
    \item[(a)] $\mc{W}\m[(\mc{B}\n\mc{W})^{p}]^{D} = [(\mc{W}\m\mc{B})^{p}]^{D}\n\mc{W}$,
    \item[(b)] $\mc{B}\n[(\mc{W}\m\mc{B})^{p}]^{D} = [(\mc{B}\n\mc{W})^{p}]^{D}\m\mc{B}.$
\end{enumerate}
\end{theorem}

\begin{proof}
By using method of induction, we will claim the first part. Let us assume $p=1.$ Now by Lemma  \ref{lem4.1} $(a)$ and Theorem \ref{mr1.2} $(a)$, we get $\mc{W}\m[(\mc{B}\n\mc{W})]^{D}
= \mc{W}\m\mc{B}\n[(\mc{W}\m\mc{B})^{2}]^{D}\n\mc{W} = (\mc{W}\m\mc{B})^{D}\n\mc{W}\m\mc{B}\n(\mc{W}\m\mc{B})^{D}\n\mc{W} = (\mc{W}\m\mc{B})^{D}\n\mc{W}$. Thus the claim is true for $p=1.$ Assume it is true for $p = k.$ That is $\mc{W}\m[(\mc{B}\n\mc{W})^{k}]^{D} = [(\mc{W}\m\mc{B})^{k}]^{D}\n\mc{W}$. Next we will claim  for $p=k+1.$ Since 
\begin{eqnarray*}
  \mc{W}\m[(\mc{B}\n\mc{W})^{k+1}]^{D} &=& \mc{W}\m[(\mc{B}\n\mc{W})^{k}]^{D}\n(\mc{B}\n\mc{W})^{D}  = [(\mc{W}\m\mc{B})^{k}]^{D}\m\mc{W}\m(\mc{B}\n\mc{W})^{D}\\
   &=& [(\mc{W}\m\mc{B})^{k}]^{D}\n(\mc{W}\m\mc{B})^{D}\n\mc{W}=[(\mc{W}\m\mc{B})^{D}]^{k}\n(\mc{W}\m\mc{B})^{D}\n\mc{W}\\
 &=& [(\mc{W}\m\mc{B})^{D}]^{k+1}\n\mc{W} = [(\mc{W}\m\mc{B})^{k+1}]^{D}\n\mc{W}.
\end{eqnarray*}
Therefore, $\mc{W}\m[(\mc{B}\n\mc{W})^{p}]^{D} = [(\mc{W}\m\mc{B})^{p}]^{D}\n\mc{W},$ $p\in\mathbb{N}.$  Using the similar lines, we can show $\mc{B}\n[(\mc{W}\m\mc{B})^{p}]^{D} = [(\mc{B}\n\mc{W})^{p}]^{D}\m\mc{B}$ for all $p\in\mathbb{N}.$
\end{proof}

We next present another characterization of the W-weighted Drazin inverse of tensors.

\begin{theorem}\label{thm4.4}
Let $\mc{B} \in \mathbb{C}^{\textbf{I}(M) \times \textbf{J}(N)}$ and  $\mc{W} \in
\mathbb{C}^{\textbf{J}(N) \times\textbf{I}(M)}.$ Then for every positive integer $p$, there exist an unique tensor $\mc{X}\in \mathbb{C}^{\textbf{I}(M)\times \textbf{J}(N)}$ such that 
\begin{enumerate}
    \item[(a)] $(\mc{B}\n\mc{W})^{D}\m\mc{X}\n\mc{W} = [(\mc{B}\n\mc{W})^{p}]^{D},$
    \item[(b)] $\mc{B}\n\mc{W}\m\mc{X} = \mc{X}\n\mc{W}\m\mc{B}, $
    \item[(c)] $\mc{B}\n\mc{W}\m(\mc{B}\n\mc{W})^{D}\m\mc{X} = \mc{X}.$ Further, $\mc{X} = \mc{B}\n[(\mc{W}\m\mc{B})^{p}]^{D}.$
\end{enumerate}
\end{theorem}
\begin{proof}
Let $\mc{X} = \mc{B}\n[(\mc{W}\m\mc{B})^{p}]^{D},$ we first claim that $\mc{X}$ satisfies all the three equations and then will discuss the uniqueness. Let $\mc{X} = \mc{B}\n[(\mc{W}\m\mc{B})^{p}]^{D}.$ Now by Theorem \ref{w1.2}, we have
\begin{eqnarray*}
(\mc{B}\n\mc{W})^{D}\m\mc{X}\n\mc{W}&=&(\mc{B}\n\mc{W})^{D}\m\mc{B}\n\left[(\mc{W}\m\mc{B})^{p}\right]^{D}\n\mc{W}\\
&=&(\mc{B}\n\mc{W})^{D}\m\left[(\mc{B}\n\mc{W})^{D}\right]^{p}\m\mc{B}\n\mc{W}=\left[(\mc{B}\n\mc{W})^{D}\right]^{p+1}\m\mc{B}\n\mc{W}\\
&=&\left[(\mc{B}\n\mc{W})^{D}\right]^{p}=\left[(\mc{B}\n\mc{W})^{p}\right]^{D},
\end{eqnarray*}
\begin{eqnarray*}
\textnormal{and ~~} \mc{B}\n\mc{W}\m\mc{X} &=& \mc{B}\n\mc{W}\m\mc{B}\n\left[\left(\mc{W}\m\mc{B}\right)^p\right]^D=\mc{B}\n\mc{W}\m\left[\left(\mc{B}\n\mc{W}\right)^p\right]^D\m\mc{B}\\
 &=&\mc{B}\n\left[\left(\mc{W}\m\mc{B}\right)^p\right]^D\n\mc{W}\m\mc{B}=\mc{X}\n\mc{W}\m\mc{B},
\end{eqnarray*}
and 
$\mc{B}\n\mc{W}\m(\mc{B}\n\mc{W})^{D}\m\mc{X} =\mc{B}\n\mc{W}\m(\mc{B}\n\mc{W})^{D}\m\mc{B}\n[(\mc{W}\m\mc{B})^{p}]^{D}=[(\mc{B}\n\mc{W})^{p}]^{D}\m\mc{B}=\mc{B}\n\left[(\mc{W}\m\mc{B})^p\right]^D= \mc{X}.$ Hence $\mc{X}$ satisfies all the required equations. Let $\mc{Y}$ be another tensor which satisfies $(a)-(c).$ Now $\mc{Y}=\mc{B}\n\mc{W}\m(\mc{B}\n\mc{W})^{D}\m\mc{Y}=(\mc{B}\n\mc{W})^{D}\m\mc{B}\n\mc{W}\m\mc{Y}=(\mc{B}\n\mc{W})^{D}\m\mc{Y}\n\mc{W}\m\mc{B}=[(\mc{B}\n\mc{W})^{p}]^{D}\m\mc{B}=\mc{B}\n[(\mc{W}\m\mc{B})^{p}]^{D}=\mc{X}.$  
\end{proof}

Using the method as in the proof of Theorem \ref{thm4.4}, one can prove the next theorem.
\begin{theorem}\label{thm4.5}
Let $\mc{B} \in \mathbb{C}^{\textbf{I}(M)\times \textbf{J}(N)}$ and  $\mc{W} \in
\mathbb{C}^{\textbf{J}(N) \times \textbf{I}(M)}.$ Then for every positive integer $p$, there exists an unique tensor $\mc{X}\in \mathbb{C}^{\textbf{I}(M) \times\textbf{J}(N)}$ such that 
\begin{enumerate}
    \item[(a)] $\mc{X}\n\mc{W} =  \mc{B}\n\mc{W}\m[(\mc{B}\n\mc{W})^{p}]^{D},$
    \item[(b)] $\mc{W}\m\mc{X} = \mc{W}\m\mc{B}\n[(\mc{W}\m\mc{B})^{p}]^{D}, $
    \item[(c)] $\mc{X}\n\mc{W}\m(\mc{B}\n\mc{W})^{p-1}\m\mc{X} = \mc{X}.$ Further, $\mc{X} = \mc{B}\n[(\mc{W}\m\mc{B})^{p}]^{D}.$
\end{enumerate}
\end{theorem}

By combining Theorem \ref{thm4.4} and Theorem \ref{thm4.5} for a particular choice of $p=2$, we get the following result as a corollary.

\begin{corollary}\label{cor4.6}
Let  $\mc{B} \in \mathbb{C}^{\textbf{I}(M) \times \textbf{J}(N)}$ and  $\mc{W} \in
\mathbb{C}^{\textbf{J}(N) \times \textbf{I}(M)}.$ Then the tensor $\mc{X} = \mc{B}\n[(\mc{W}\m\mc{B})^{2}]^{D}$ is the W-weighted Drazin inverse of $\mc{B}.$
\end{corollary}

The above Corollary reflected the existence of the W-weighted drazin inverse and the uniqueness of the W-weighted Drazin inverse is discussed in the next theorem.         

\begin{theorem}\label{w1.4}
The tensor $\mc{X} = \mc{B}\n[(\mc{W}\m\mc{B})^{2}]^{D}$ is the unique solution of the following tensor equations
\begin{enumerate}
    \item[(a)] $(\mc{B}\n\mc{W})^{k+1}\m\mc{X}\n\mc{W}=(\mc{B}\n\mc{W})^{k}$,
    \item[(b)] $ \mc{X} = \mc{X}\n\mc{W}\m\mc{B}\n\mc{W}\m\mc{X}$,
    \item[(c)]  $\mc{B}\n\mc{W}\m\mc{X} = \mc{X}\n\mc{W}\m\mc{B}.$
\end{enumerate}
\end{theorem}

\begin{proof}
The existence of solution is trivially holds by Corollary \ref{cor4.6}. It is enough to show only the uniqueness of $\mc{X}.$ Suppose there exists another tensor  $\mc{Y}$  which satisfies the conditions $(a)-(c).$ Now, we have 
\begin{eqnarray*}
  \mc{X}& =& (\mc{X}\n\mc{W}\m\mc{B})\n\mc{W}\m\mc{X} =  (\mc{B}\n\mc{W})\m\mc{X}\n\mc{W}\m\mc{X}\\
  &=& (\mc{B}\n\mc{W})^{2}\m(\mc{X}\n\mc{W})^{2}\m\mc{X}= \cdots\\
  &=&(\mc{B}\n\mc{W})^{{k}}\m(\mc{X}\n\mc{W})^{{k}}\m\mc{X} = (\mc{B}\n\mc{W})^{{k}+1}\m\mc{Y}\n\mc{W}\m(\mc{X}\n\mc{W})^{{k}}\m\mc{X}\\
  &=& \mc{Y}\n(\mc{W}\m\mc{B})^{{k}+1}\n\mc{W}\m(\mc{X}\n\mc{W})^{{k}}\m\mc{X}\\
  &=& \mc{Y}\n\mc{W}\m\mc{B}\n\mc{W}\m(\mc{B}\n\mc{W})^{{k}}\m(\mc{X}_1\n\mc{W})^{{k}}\m\mc{X}= \mc{Y}\n\mc{W}\m(\mc{B}\n\mc{W}\m\mc{X})\\
  &=& \mc{Y}\n\mc{W}\m\mc{X}\n\mc{W}\m\mc{B} = \mc{Y}\n\mc{W}\m\mc{B}\n\mc{W}\m\mc{Y}\n\mc{W}\m\mc{X}\n\mc{W}\m\mc{B}\\
  &=& \mc{Y}\n\mc{W}\m\mc{B}\n\mc{W}\m\mc{Y}\n\mc{W}\m\mc{B}\n\mc{W}\m\mc{Y}\n\mc{W}\m\mc{X}\n\mc{W}\m\mc{B}\\
  &=& \mc{Y}\n\mc{W}\m\mc{Y}\n\mc{W}\m\mc{B}\n\mc{W}\m\mc{Y}\n\mc{W}\m\mc{B}\n\mc{W}\m\mc{X}\n\mc{W}\m\mc{B}\\
  &=& \mc{Y}\n\mc{W}\m\mc{Y}\n\mc{W}\m\mc{Y}\n\mc{W}\m\mc{B}\n\mc{W}\m\mc{B}\n\mc{W}\m\mc{X}\n\mc{W}\m\mc{B}\\
  &=& \mc{Y}\n\mc{W}\m\mc{Y}\n\mc{W}\m\mc{Y}\n\mc{W}\m(\mc{B}\n\mc{W})^{2}\m\mc{X}\n\mc{W}\m\mc{B}\\
  &=& \mc{Y}\n\mc{W}\m\mc{Y}\n\mc{W}\m\mc{Y}\n\mc{W}\m\mc{B}\n\mc{W}\m\mc{B}\\
  &=& \mc{Y}\n\mc{W}\m\mc{Y}\n\mc{W}\m\mc{B}\n\mc{W}\m\mc{Y}\n\mc{W}\m\mc{B} = \mc{Y}\n\mc{W}\m\mc{Y}\n\mc{W}\m\mc{B}\\
  &=& \mc{Y}\n\mc{W}\m\mc{B}\n\mc{W}\m\mc{Y}
  = \mc{Y}.
\end{eqnarray*}
Therefore, $\mc{X} = \mc{B}\n[(\mc{W}\m\mc{B})^{2}]^{D}$ is the unique solution.
\end{proof}

We conclude this section with an additional property of the W-weighted Drazin inverse, which helps to compute the W-weighted Drazin inverse via index one tensors.

\begin{theorem}\label{w1.5}
Let $\mc{B}, \mc{X} \in \mathbb{C}^{\textbf{I}(M) \times \textbf{J}(N)}$. Then $\mc{X} = \mc{B}\n[(\mc{W}\m\mc{B})^{2}]^{D}$, for some tensor $\mc{W}$
if and only if $\mc{X} = \mc{B}\n\mc{Y}\m\mc{B}\n\mc{Y}\m\mc{B}$
for some tensor $\mc{Y} \in
\mathbb{C}^{\textbf{J}(N) \times \textbf{I}(M)}$
with ind\,$(\mc{B}\n\mc{Y})=1=$ ind\,$(\mc{Y}\m\mc{B})$.
\end{theorem}

\begin{proof}
Let $\mc{X} = \mc{B}\n[(\mc{W}\m\mc{B})^{2}]^{D}.$ By using Theorem \ref{w1.2} and Theorem \ref{mr1.2} $(a)$, we obtain
\begin{eqnarray*}
   \mc{X} &=& \mc{B}\n[(\mc{W}\m\mc{B})^{2}]^{D} = \mc{B}\n(\mc{W}\m\mc{B})^{D}\n(\mc{W}\m\mc{B})^{D}= (\mc{B}\n\mc{W})^{D}\m\mc{B}\n(\mc{W}\m\mc{B})^{D}\\
  & =& (\mc{B}\n\mc{W})\m[(\mc{B}\n\mc{W})^{2}]^{D}\m\mc{B}\n[(\mc{W}\m\mc{B})^{2}]^{D}\n(\mc{W}\m\mc{B})\\
   & =& \mc{B}\n[(\mc{W}\m\mc{B})^{2}]^{D}\n\mc{W}\m\mc{B}\n[(\mc{W}\m\mc{B})^{2}]^{D}\n\mc{W}\m\mc{B}=\mc{B}\n\mc{Y}\m\mc{B}\n\mc{Y}\m\mc{B},
\end{eqnarray*}
where $\mc{Y} = [(\mc{W}\m\mc{B})^{2}]^{D}\n\mc{W}.$ Next we will show $\mc{B}\n\mc{Y}$ and  $\mc{Y}\m\mc{B}$ are of index one. As
$ \mc{B}\n\mc{Y}= \mc{B}\n[(\mc{W}\m\mc{B})^{2}]^{D}\n\mc{W} =[(\mc{B}\n\mc{W})^{2}]^{D}\m\mc{B}\n\mc{W}=(\mc{B}\n\mc{W})^{D} $ and   
$\mc{Y}\m\mc{B} =[(\mc{W}\m\mc{B})^{2}]^{D}\n\mc{W}\m\mc{B}= (\mc{W}\m\mc{B})^{D}.$ Therefore, both have index one. Which completes the sufficient part.

Conversely, let $\mc{X} = \mc{B}\n\mc{Y}\m\mc{B}\n\mc{Y}\m\mc{B}$ and $\mc{W} = \mc{Y}\m[(\mc{B}\n\mc{Y})^{2}]^{D} = [(\mc{Y}\m\mc{B})^{2}]^{D}\n\mc{Y}$ such that ind$(\mc{B}\n\mc{Y})=\mbox{ind}(\mc{Y}\m\mc{B})=1.$  To claim the necessary part, it is enough to show, $\mc{X}$ satisfies all the assumptions of Theorem \ref{thm4.5} for $p=2.$  Notice that, from $\mc{W},$ we easily get $(\mc{W}\m\mc{B})^D=\mc{Y}\m\mc{B}$ and $(\mc{B}\n\mc{W})^D=\mc{B}\n\mc{Y}$, since $\mc{B}\n\mc{Y}$ and $\mc{Y}\m\mc{B}$ are of index one.
Now $\mc{X}\n\mc{W}=\mc{B}\n\mc{Y}\m\mc{B}\n\mc{Y}\m\mc{B}\n\mc{Y}\m[(\mc{B}\n\mc{Y})^{2}]^{D}=(\mc{B}\n\mc{Y})^3*_N[(\mc{B}\n\mc{Y})^D]^2=\mc{B}\n\mc{Y}=(\mc{B}\n\mc{W})^D=\mc{B}\n\mc{W}\m[(\mc{B}\n\mc{W})^2]^D.$ Similarly, $\mc{W}\m\mc{X} = (\mc{W}\m\mc{B})\n\mc{Y}\m\mc{B}\n\mc{Y}\m\mc{B}
= (\mc{Y}\m\mc{B})^{D}\n\mc{Y}\m\mc{B}\n\mc{Y}\m\mc{B} = \mc{Y}\m\mc{B} = (\mc{W}\m\mc{B})^{D}=\mc{W}\m\mc{B}\n[(\mc{W}\m\mc{B})^{2}]^D.$ Further, we have $\mc{X}\n\mc{W}\m\mc{B}\n\mc{W}\m\mc{X} = \mc{B}\n\mc{Y}\m\mc{B}\n\mc{Y}\m\mc{B} = \mc{X}.$ 
\end{proof}

\section{Multilinear system}
The main objective of this section is solving multilinear systems. In the first part, we discuss the solution of the singular multilinear system using the Drazin inverse of tensor, and in the second part, we address iterative method (higher order Gauss-Seidel) and its convergence analysis, for solving high-dimensional Poisson problems in the multilinear system framework.

\subsection{Drazin-inverse solution} 
Let $\mc{A} \in \mathbb{C}^{\textbf{I}(N) \times \textbf{I}(N) }$ and consider the following singular tensor equation 
\begin{equation}\label{eq1.10}
   \mc{A}\n\mc{X} = \mc{B}, ~~ \mc{X},~ \mc{B} \in \mathbb{C}^{\textbf{I}(N) }.
\end{equation}
If the tensor $\mc{B}\in \mathfrak{R}(\mc{A}^k),$ then Eq. (\ref{eq1.10}) is called the Drazin consistent multilinear system and its solution, we call Drazin-inverse solution or simply solution. 
Such multilinear systems arise in numerous applications in computational science and engineering such as continuum physics and engineering, isotropic
and anisotropic elasticity \cite{lai}. Multilinear systems are also prevalent in solving PDEs numerically. Let us recall the lemma of Drazin consistent multilinear system which was proved in \cite{Wei18}, very recently.

 \begin{lemma}[Lemma 5.1, \cite{Wei18}]\label{wei1.2}
Let $\mc{A}\in\mathbb{C}^{\textbf{I}(N)\times\textbf{I}(N)}$ and $ind(\mc{A})=k$. Then  $\mc{A}^{D}\n\mc{B}$ is a solution of $(\ref{eq1.10})$ if and only if $\mc{B} \in \mathfrak{R}(\mc{A}^k)$. 
\end{lemma}

In addition to this, the authors of \cite{Wei18} discussed the general solution of \eqref{eq1.10}, as follows.

\begin{theorem}[Theorem 5.2, \cite{Wei18}]\label{wei1.3}
Let $\mc{A}\in\mathbb{C}^{\textbf{I}(N)\times\textbf{I}(N)}$ and $ind(\mc{A})=k$. If $\mc{B} \in \mathfrak{R}(\mc{A}^k),$ then the general solution of $(\ref{eq1.10})$ is of the form $\mc{X} = \mc{A}^{D}\n\mc{B}+(I-\mc{A}^{D}\n\mc{A})\n\mc{Z}$, for any arbitrary tensor $\mc{Z}\in \mathfrak{R}(\mc{A}^{k-1})+\mc{N}(\mc{A})$.
\end{theorem}
 
We now show the existence and uniqueness of the Drazin inverse solution in the following theorem.

\begin{lemma}\label{lemma5.1}
Let  $\mc{A} \in \mathbb{C}^{\textbf{I}(N) \times \textbf{I}(N) }$ and ind$(\mc{A})=k.$ If $\mc{X}\in\mathfrak{R}(\mc{A}^k),$ then the singular tensor equation
\begin{equation*}
    \mc{A}\n\mc{X} = \mc{B},
\end{equation*}
has one and only one solution, and is given by 
$\mc{X} = \mc{A}^{D}\n\mc{B}.
$
\end{lemma}
\begin{proof}
Let $\mc{X} \in \mathfrak{R}(\mc{A}^k).$ So there exists a tensor $\mc{Y} \in \mathbb{C}^{\textbf{I}(N) }$ such that  $\mc{X} = \mc{A}^{k}\n\mc{Y}.$ Now using Definition \ref{dz1.1}, we get
\begin{equation*}
\mc{X}= \mc{A}^k\n\mc{Y} =\mc{A}^{k+1}\n\mc{A}^{D}\n\mc{Y}= \mc{A}^{D}\n\mc{A}^{k+1}\n\mc{Y}=\mc{A}^{D}\n\mc{A}\n\mc{X}=\mc{A}^{D}\n\mc{B}.  
\end{equation*}
Further, in view of the Lemma \ref{wei1.2} and Theorem \ref{wei1.1} $(a)$, we obtain $\mc{X}-\mc{A}^{D}\n\mc{B}\in  \mathfrak{R}(\mc{A}^k)$. Again by Theorem \ref{wei1.3} and Theorem \ref{wei1.1} $(b)$ we get $\mc{X}-\mc{A}^{D}\n\mc{B}\in \mc{N}(\mc{A}^k)$. Hence $\mc{X}-\mc{A}^{D}\n\mc{B}\in \mathfrak{R}(\mc{A}^k)\cap \mc{N}(\mc{A}^k) = \{0\}$. Therefore, the solution  $\mc{A}^{D}\n\mc{B}$ is unique.
\end{proof}

In case of index one  $(k=1),$  the result is stated in the next corollary.

\begin{corollary}\label{cor1.12}
Let $\mc{A} \in \mathbb{C}^{\textbf{I}(N) \times \textbf{I}(N) }$ and ind$(\mc{A})=1.$ If $\mc{X} \in \mathfrak{R}(\mc{A}),$ then singular tensor Eq. (\ref{eq1.10}) 
has  unique solution, and is given by 
$ \mc{X} = \mc{A}^{\#}\n\mc{B}.$
\end{corollary}

We now discuss some results concerning the equivalent multilinear systems of the Drazin inverse of tensors.  In particular, the relationship between the solutions of the multilinear system \eqref{eq1.10} and the following tensor-based Drazin normal equation, 
\begin{equation}\label{mt1.1}
\mc{A}^{k+1}\n\mc{X} = \mc{A}^{k}\n\mc{B}.   
\end{equation}
is analyzed, i.e., if $\mc{B}\in\mathfrak{R}(\mc{A}^k)$ and ind$(\mc{A})=k,$ then it is easy to verify that,  each solution of Eq. \eqref{eq1.10} is also a solution of Eq. \eqref{mt1.1} and vice versa. In spite of this fact,
we discuss the solution of the Drazin normal equation for tensor, as follows. 

\begin{theorem}\label{mrt1.1}
Let  $\mc{A} \in \mathbb{C}^{\textbf{I}(N) \times \textbf{I}(N) }$, $\mc{B}\in\mathfrak{R}(\mc{A}^k)$ and ind$(\mc{A})=k.$ Then the set of all solutions of Eq. $(\ref{mt1.1})$ is given by
\begin{equation*}
   \mc{X} = \mc{A}^{D}\n\mc{B}+\mc{N}({\mc{A}^{k}}). 
\end{equation*}
Moreover, $\mc{X} = \mc{A}^{D}\n\mc{B}$ is the unique solution of $(\ref{mt1.1})$ in $\mathfrak{R}(\mc{A}^k)$.
\end{theorem}

\begin{proof}
From the definition of the Drazin inverse of a tensor  $\mc{A} \in \mathbb{C}^{\textbf{I}(N) \times \textbf{I}(N) }$, we have
\begin{equation*}
\mc{A}^{k+1}*_N(\mc{X}-\mc{A}^{D}\n\mc{B}) = \mc{A}^{k}\n\mc{A}\n\mc{X}-\mc{A}^{k}\n\mc{A}\n\mc{A}^{D}\n\mc{B} = \mc{A}^{k}\n\mc{B}-\mc{A}^{k}\n\mc{B} = \mc{O}.
\end{equation*}
Using the Theorem \ref{weitm3.2}, we obtain $\mc{X}-\mc{A}^{D}\n\mc{B}\in \mc{N}({\mc{A}^{k+1}}) = \mc{N}({\mc{A}^{k}})$. Hence $\mc{X} = \mc{A}^{D}\n\mc{B}+\mc{N}({\mc{A}^{k}}).$ To show the uniqueness in $\mathfrak{R}(\mc{A}^k),$ let  $\mc{U}\in\mathfrak{R}(\mc{A}^k)$ be any solution of Eq. (\ref{mt1.1}).  Now $\mc{U}-\mc{A}^{D}\n\mc{B}\in\mathfrak{R}(\mc{A}^k)$, and $\mc{A}^{k+1}\n\mc{U}-\mc{A}^{k+1}\n\mc{A}^{D}\n\mc{B} = \mc{A}^{k}\n\mc{B}-\mc{A}^{k}\n\mc{B} = \mc{O}$. Thus  $\mc{U}-\mc{A}^{D}\n\mc{B}\in\mc{N}({\mc{A}^{k}})$. Therefore, $\mc{U}-\mc{A}^{D}\n\mc{B}\in\mathfrak{R}(\mc{A}^k)\cap\mc{N}({\mc{A}^{k}}) = \{0\}$. Hence completes the proof. 
\end{proof}
One can find the matrix version of the above result in \cite{milj,wei16}. Further, we discuss the Drazin solution of another normal equation, called modified Drazin normal equation and is defined by the following tensor equation 
\begin{equation}\label{eq1.30}
 \mc{A}^{2k}\n\mc{X} = \mc{A}^{k}\n\mc{B},~\mc{A} \in \mathbb{C}^{\textbf{I}(N) \times \textbf{I}(N) },~\mc{X},~\mc{B} \in \mathbb{C}^{\textbf{I}(N) }. 
\end{equation}

\begin{theorem}\label{m.18}
Let  $\mc{A} \in \mathbb{C}^{\textbf{I}(N) \times \textbf{I}(N) }$, $\mc{B}\in\mathfrak{R}(\mc{A}^k)$ and ind$(\mc{A})=k.$ Then, the set of all solutions of Eq. (\ref{eq1.30}) is given by
\begin{equation*}
 \mc{X} = (\mc{A}^{k})^{D}\n\mc{B}+\mc{N}(\mc{A}^{k}).  
\end{equation*}
\end{theorem}

\begin{proof}
The tensor equation (\ref{eq1.30}) is alwasy consistent, since  $\mc{A}^{k}\n\mc{B} \in \mc{R}(\mc{A}^{k}) = \mc{R}(\mc{A}^{2k})$. By Theorem \ref{mr1.2} (a) and Corollary \ref{mr1.8}, we have $(\mc{A}^{k})^{D}\n\mc{B} = (\mc{A}^{k})^{\#}\n\mc{B} \in \mc{R}(\mc{A}^{k}) $. Now $\mc{A}^{2k}*_N(\mc{X}-(\mc{A}^{k})^{D}\n\mc{B}) = \mc{A}^{2k}\n\mc{X}-\mc{A}^{k}\n\mc{A}^{k}\n(\mc{A}^{k})^{\#}\n\mc{B} = \mc{A}^{k}\n\mc{B}-\mc{A}^{k}\n\mc{B} = \mc{O}$. Hence $\mc{X}-(\mc{A}^{k})^{D}\n\mc{B}\in\mc{N}({\mc{A}^{2k}}) = \mc{N}({\mc{A}^{k}})$. So $\mc{X} = (\mc{A}^{k})^{D}+\mc{N}({\mc{A}^{k}})$ is the solution of $(\ref{eq1.30})$. 
\end{proof}
Using the method as in the proof of Theorem \ref{mrt1.1}, one can show  $(\mc{A}^{k})^{D}\n\mc{B}$ is the unique solution of Eq. (\ref{eq1.30}) in $\mathfrak{R}(\mc{A}^k).$ 

In the following, we present an example to illustrate our result. In order to show how the Drazin inverse of
tensors are employed in the two dimensional Poisson problem. 

\begin{example}\label{pde1}
Consider the following partial differential equation
\begin{equation}\label{PDE11}
    \frac{\partial^2 u}{\partial x^2}+\frac{\partial^2 u}{\partial y^2}=f(x,y),~(x,y)\in \Omega=[0,1]\times [0,1]
\end{equation}
 with Neumann boundary conditions. If we apply 5-point stencil central difference scheme on a uniform grid with $m^2$ nodes, we obtain the following tensor equation 
 \begin{equation}\label{1stexam}
     \mc{A}*_2\mc{X}=\mc{B}, ~\mc{X}=(u_{kl})\in\mathbb{R}^{m\times m} \textnormal{~and~}  \mc{B}=(b_{ij})\in\mathbb{R}^{m\times m}, 
 \end{equation}
and the tensor $\mc{A}=(a_{ijkl})\in\mathbb{R}^{m\times m\times m\times m}$ is of the form 
\begin{equation}\label{F-ex-ten}
\mc{A}=\mc{I}_m\kronecker \mc{P} +\mc{Q}\kronecker \mc{I}_m+\mc{D},
\end{equation}
where $\mc{I}_m \in \mathbb{R}^{m\times m}$ is the second order identity tensor. The second order tensors $\mc{P}\in\mathbb{R}^{m\times m}$ and $\mc{Q}\in\mathbb{R}^{m\times m}$ are of the form
\begin{equation*}
    \mc{P}=tridiagonal\left(-1,0,-1\right) = \left( \begin{array}{cccc}
0 & -1 & & 0\\
-1 & \ddots & \ddots & \\
& \ddots & \ddots & -1 \\
0 & & -1 & 0 \end{array} \right)=
 \mc{Q}.
\end{equation*}
Further, the tensor $\mc{D}\in\mathbb{R}^{m\times m\times m\times m}$ is a diagonal tensor, where the diagonal elements will change with respect to number of grid points.   
From the representation \eqref{F-ex-ten} (the coefficient tensor $\mc{A}$), it is clear that $ind(\mc{A}) = 1$. Thus the solution of the multilinear system \eqref{1stexam} becomes $\mc{X}=\mc{A}^{\#}*_2\mc{B}.$  
We consider a tensor $\mc{B}$ from $\mathfrak{R}(A),$ and calculate the approximate solution of the partial differential equation \eqref{PDE11} with  different choices of $m$, which are presented in Figure 1.
\end{example}

\begin{figure}[t!]\label{exa-1}
\begin{center}
\begin{tabular}{cc}
\subfigure[] { \includegraphics[width=0.43\textwidth]{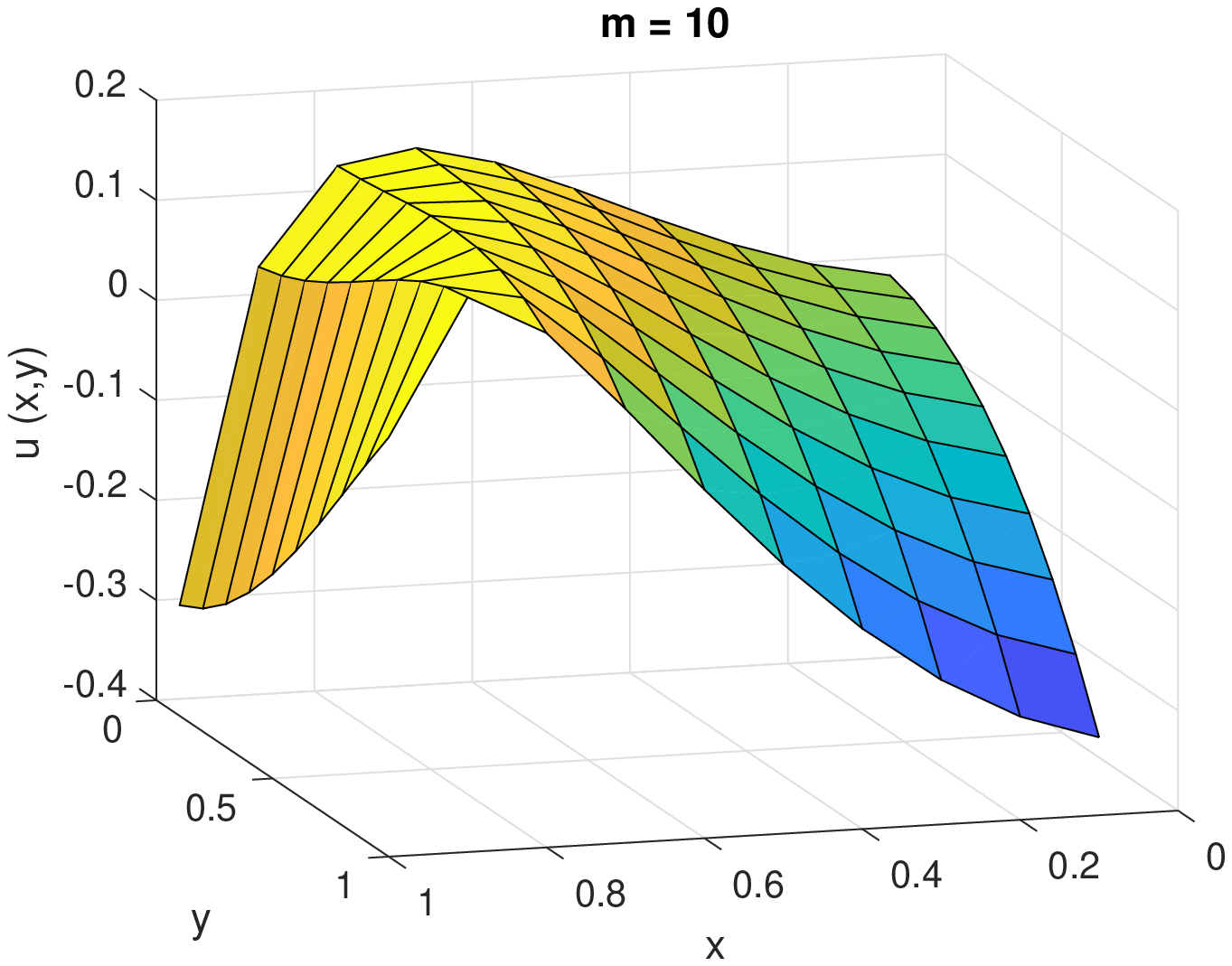}}&
\subfigure[] { \includegraphics[width=0.43\textwidth]{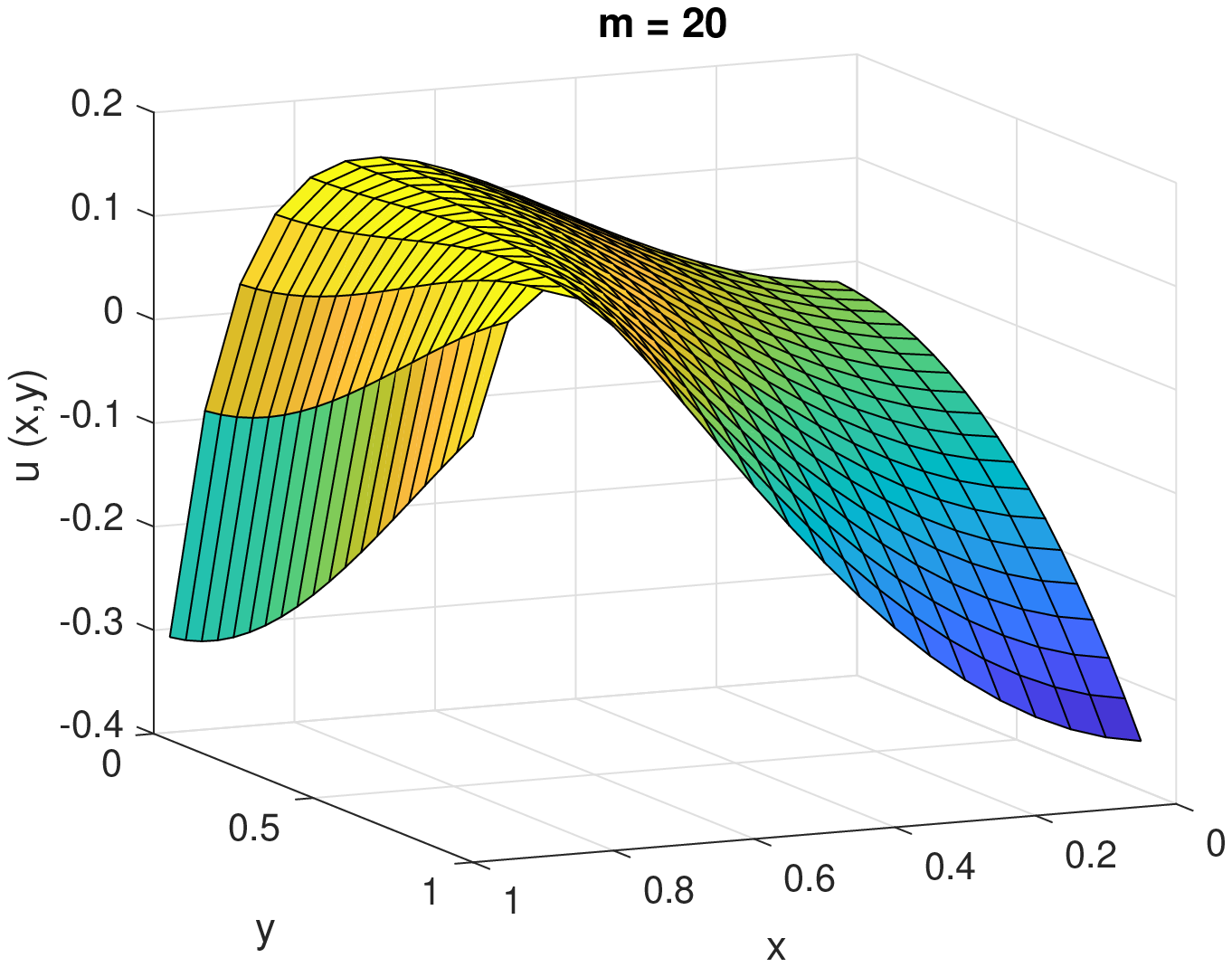}}\\
\subfigure[] { \includegraphics[width=0.43\textwidth]{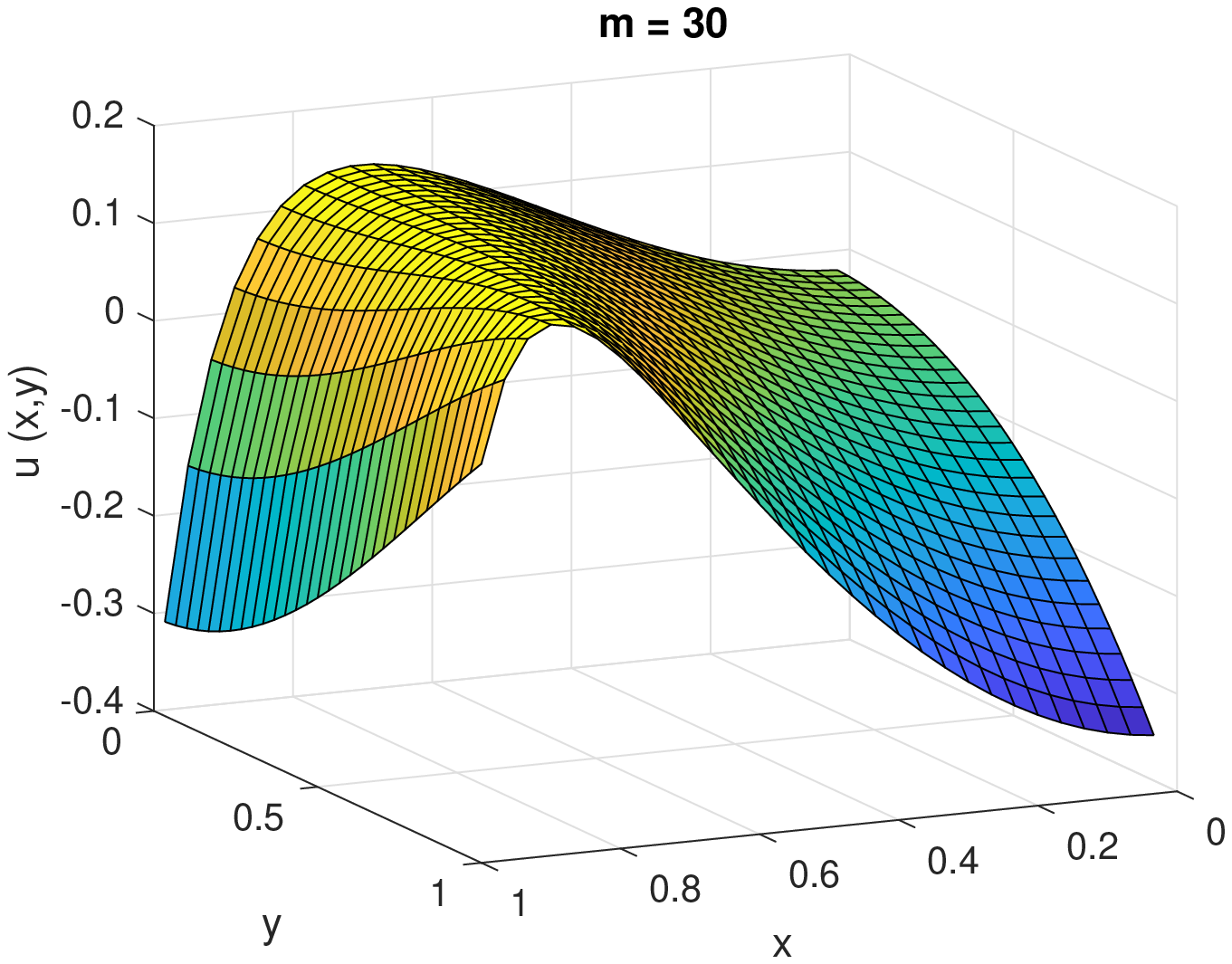}}&
\subfigure[] { \includegraphics[width=0.43\textwidth]{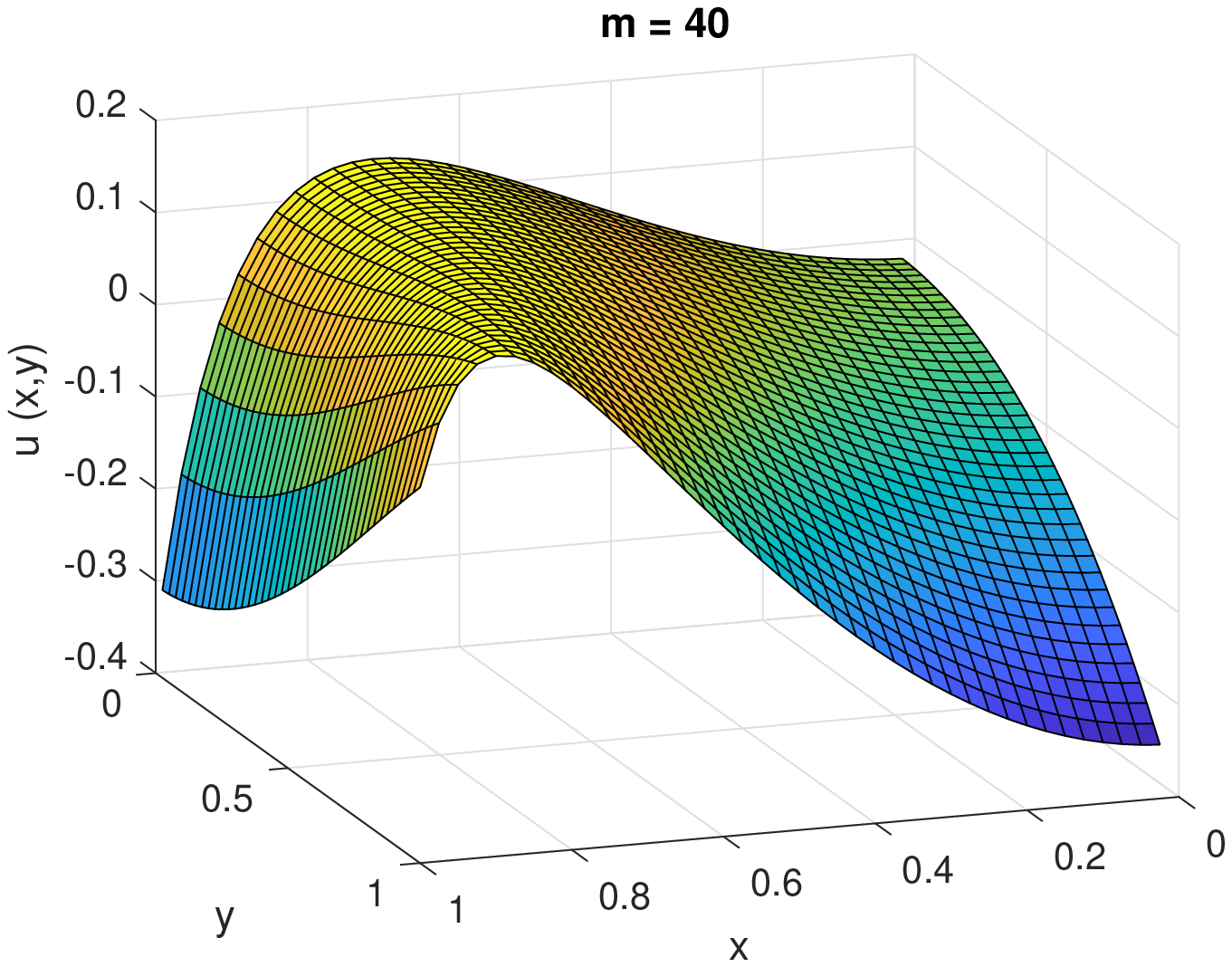}}
\end{tabular}
\caption{Solution of the multilinear system for different values of $m$.}
\label{finalbell}
\end{center}
\end{figure}

\subsection{Iterative method} 

Recently there has been increasing interest in developing
the tensor-based iterative method for solving multilinear systems \cite{bral,li17,li18, liang18,liu18}.  In case of nonsingular and positive definite tensor, few iterative schemes such as Jacobi and biconjugate gradient (BiCG) are discussed in \cite{bral} with help of Einstein product. 
A general form of tensor-based iterative method for the multilinear system \eqref{eq1.10} is defined as
\begin{eqnarray}\label{iteq}
   \mc{X}^{(k+1)} = \mc{H}\n \mc{X}^{(k)} + \mc{C}, ~~~ \textnormal{for} ~~ k=0,1,2, \cdots
\end{eqnarray}
where $\mc{X}^{(k+1)}$ and $\mc{X}^{(k)}$are the approximations for the tensor $\mc{X}$ at the $(k+1)$-th and $k$-th iteration, respectively. Here $\mc{H}$ is called the iteration tensor depending on $\mc{A}$ and $\mc{C}$. In case of limiting, when $k \rightarrow \infty,$ $\mc{X}^{(k)}$ converges to the exact solution 
\begin{eqnarray*}
   \mc{X} = \mc{A}^{-1}*_N \mc{B}.
\end{eqnarray*}
Now the iteration scheme \eqref{iteq} becomes,
$\mc{A}^{-1}*_N \mc{B} = \mc{H}*_N\mc{A}^{-1}*_N \mc{B} +\mc{C},$ which implies $\mc{C} = (\mc{I} - \mc{H})*_N \mc{A}^{-1}*_N\mc{B}$. Consider $\mc{A}=\mc{L}+\mc{D}+\mc{U},$ where $\mc{L}$ is the lower off-diagonal tensor, $\mc{D}$ is the diagonal tensor and $\mc{U}$ is the upper off-diagonal tensor. Then the Gauss-Seidel iteration method gives the iteration tensor $\mc{H} = - (\mc{D} + \mc{L})^{-1}*_N \mc{U}$  and $\mc{C} = (\mc{D} + \mc{L})^{-1} *_N \mc{B}$.   Thus our aim is to analyze the iteration tensor $\mc{H}$.

Now, we recall the Frobenius norm $|| \cdot||_F$  of a tensor $\mc{A}\in\mathbb{C}^{\textbf{I}(N)\times \textbf{J}(N)}$ which was introduced in \cite{sun}, as follows.
\begin{equation*}
\|\mc{A}\|_F=\left(\sum_{\textbf{i}(N),\textbf{j}(N)}\left|a_{\textbf{i}(N),\textbf{j}(N)}\right|^2\right)^{1/2}.
\end{equation*}
With reference to the Frobenius norm $|| \cdot||_F$, we define the maximum norm
\begin{eqnarray*}
||\cdot  ||_\infty = \max_{\textbf{j}(N)}\left( \displaystyle\sum_{\textbf{i}(N)}^{\mathbf{\hat{I}}(N)}\left|a_{\textbf{i}(N),\textbf{j}(N)}\right|\right).
\end{eqnarray*}
Using the above definition, we now prove the following  result on Frobenius norm.
\begin{lemma}\label{normin}
Let $\mc{A}\in\mathbb{C}^{\textbf{I}(N)\times\textbf{P}(N)}$  and  $\mc{B}\in\mathbb{C}^{\textbf{P}(N)\times\textbf{J}(N)}.$ Then
$\|\mc{A}\n\mc{B}\|_F\leq \|\mc{A}\|_F\cdot\|\mc{B}\|_F.$
\end{lemma}
\begin{proof}
By applying Cauchy-Schwarz inequality to the inner summation, we have 
\begin{eqnarray*}
 \|\mc{A}\n\mc{B}\|_F&=&\left(\sum_{\textbf{i}(N),\textbf{j}(N)}\left|{\sum_{\textbf{k}(N)}a_{\textbf{i}(N),\textbf{k}(N)}b_{\textbf{k}(N),\textbf{j}(N)}}\right|^2\right)^{1/2}\\
 &\leq& \left(\sum_{\textbf{i}(N),\textbf{j}(N)}\left[\sum_{\textbf{k}(N)}|a_{\textbf{i}(N),\textbf{k}(N)}b_{\textbf{k}(N),\textbf{j}(N)}|\right]^2\right)^{1/2}\\
 &\leq &\left(\sum_{\textbf{i}(N),\textbf{j}(N)}\left(\sum_{\textbf{k}(N)}|a_{\textbf{i}(N),\textbf{k}(N)}|^2\sum_{\textbf{k}(N)}|b_{\textbf{k}(N),\textbf{j}(N)}|^2\right)\right)^{1/2}\\
 &=&\left(\sum_{\textbf{i}(N),\textbf{k}(N)}|a_{\textbf{i}(N),\textbf{k}(N)}|^2\right)^{1/2}\cdot \left(\sum_{\textbf{k}(N),\textbf{j}(N)}|b_{\textbf{k}(N),\textbf{j}(N)}|^2\right)^{1/2}\\
 &=&\|\mc{A}\|_F\cdot\|\mc{B}\|_F.
\end{eqnarray*}
\end{proof}
In case of  $\mc{B}=\mc{A},$ we obtain the following result as a corollary.
\begin{corollary}\label{cor5.14}
Let $\mc{A}\in\mathbb{C}^{\textbf{I}(N)\times\textbf{I}(N)}.$ Then $\|\mc{A}^k\|\leq \|\mc{A}\|^k$ for any positive integer $k$.
\end{corollary}
\begin{theorem}\label{spectr}
Let $\mc{A}\in\mathbb{C}^{\textbf{I}(N)\times\textbf{I}(N)}.$  Then 
\begin{enumerate}
    \item[(a)] $\displaystyle\lim_{k\rightarrow\infty}\mc{A}^k=\mc{O}$ if $\|\mc{A}\|<1$ or if and only $\rho(\mc{A})<1$.
    \item[(b)] the series $\displaystyle\sum_{k=0}^\infty\mc{A}^k$ is convergent if and only if $\displaystyle\lim_{k\rightarrow\infty}\mc{A}^k=\mc{O}$. Further, the series converges to $(\mc{I}-\mc{A})^{-1}$.
\end{enumerate}
\end{theorem}

\begin{proof}
Let $\|\mc{A}\|<1.$ Using Corollary \ref{cor5.14}, we get $\|\displaystyle\lim_{k\rightarrow\infty}\mc{A}^k\|\leq\displaystyle\lim_{k\rightarrow\infty}\|\mc{A}\|^k=0.$ Therefore, $\displaystyle\lim_{k\rightarrow\infty}\mc{A}^k=\mc{O}.$  To show the second part of $(a),$ let $\rho(\mc{A})<1.$ In the view of the Lemma \ref{SVDTensor},  the singular value decomposition of the tensor $\mc{A}^k$ can be written as $\mc{A}^k=\mc{U}\n\mc{D}\n\mc{V}^*,$ where $\mc{U},~\mc{V}\in \mathbb{C}^{\textbf{I}(N)\times\textbf{I}(N)}$ are unitary tensors. The diagonal entries of the diagonal tensor  $\mc{D}\in\mathbb{C}^{\textbf{I}(N)\times\textbf{I}(N)}$ are the eigenvalues of $\mc{A}^k.$ Hence $\displaystyle\lim_{k\rightarrow\infty}\mc{A}^k=\mc{O}$ if and only $|d_{\textbf{ i}(N),\textbf{i}(N)}|<1.$ Thus completes $(a).$ To claim  part $(b)$, it is enough to show the necessary part since only if part is trivial from the Definition \ref{tenseries}.  Let $\displaystyle\lim_{k\rightarrow\infty}\mc{A}^k=\mc{O}.$ So by Theorem \ref{spectr} $(a),$ $\rho(\mc{A})<1.$ Thus all the eigenvalues of $(\mc{I}-\mc{A})$ are nonzero. This leads the tensor $(\mc{I}-\mc{A})$ is nonsingular. Now 
\begin{equation}\label{eq1212}
    (\mc{I}+\mc{A}+\mc{A}^2+\cdots+\mc{A}^k)\n(\mc{I}-\mc{A})=\mc{I}-\mc{A}^{k+1}.
\end{equation}
Post-multiplying Eq. \eqref{eq1212} by $
(\mc{I}-\mc{A})^{-1}$, we get 
$\mc{I}+\mc{A}+\mc{A}^2+\cdots+\mc{A}^k=(\mc{I}-\mc{A}^{k+1})\n(\mc{I}-\mc{A})^{-1}.$ By taking $k\rightarrow\infty,$ we obtain $\displaystyle\sum_{k=0}^\infty\mc{A}^k=(\mc{I}-\mc{A})^{-1}.$ \end{proof}

\begin{theorem}\label{conv}
The iterative scheme \eqref{iteq} obtained from the tensor  splitting, converges to $\mc{A}^{-1}\n\mc{B}$ for any initial guess $\mc{X}^{(0)}$ if and only if $\rho(\mc{H})< 1$. 
\end{theorem}
\begin{proof}
Without loss of generality, assume $\mc{X}^{(0)}=\mc{O}.$ Then by Eq. \eqref{iteq}, we obtain $\mc{X}^{(1)}=\mc{C}$. This leads $\mc{X}^{(2)}=\mc{H}\n\mc{X}^{(1)}+\mc{C}=\mc{H}\n\mc{C}+\mc{C}=(\mc{H}+\mc{I})\n\mc{C}.$ By succeeding $(k+1)$-times, we get 
\begin{equation*}
    \mc{X}^{(k+1)}=(\mc{I}+\mc{H}+\mc{H}^2+\cdots+\mc{H}^k)\n\mc{C}.
\end{equation*}
By taking $k\rightarrow \infty$ and applying Theorem \ref{spectr} $(b),$ we obtain $\displaystyle\lim_{k\rightarrow\infty}\mc{X}^{(k+1)}=(\mc{I}-\mc{H})^{-1}\n\mc{C}$ if and only if $\rho(\mc{H})<1.$ This is equivalently,  $\displaystyle\lim_{k\rightarrow\infty}\mc{X}^{(k+1)}=\mc{A}^{-1}\n\mc{B}$ if and only if $\rho(\mc{H})<1.$ 
\end{proof}
In view of Theorem \ref{spectr} $(a)$ and \ref{conv}, we state the following result as a corollary.
\begin{corollary}
If $\|\mc{H}\|<1,$ then the iterative scheme \eqref{iteq} converges to $\mc{A}^{-1}\n\mc{B}$ for any initial guess $\mc{X}^{(0)}.$
\end{corollary}

\begin{theorem}
If the tensor $\mc{A} \in \mathbb{C}^{\textbf{I}(N) \times \textbf{I}(N)}$ is a strictly diagonally dominant, then the Gauss-Seidel iteration scheme converges for any initial tensor $\mc{X}^{(0)}.$  
\begin{proof}
The Gauss-Seidel iteration scheme is given by 
\begin{eqnarray*}
   \mc{X}^{(k+1)} &=& - (\mc{D} + \mc{L})^{-1}*_N \mc{U}*_N\mc{X}^{(k)} +  (\mc{D} + \mc{L})^{-1} *_N \mc{B}\\
   && = - (\mc{D} + \mc{L})^{-1} *_N\left[ \mc{A} - (\mc{D} + \mc{L})\right]*_N\mc{X}^{(k)} + (\mc{D} + \mc{L})^{-1} *_N \mc{B}\\
   && = [\mc{I} - (\mc{D} + \mc{L})^{-1} \n \mc{A}]\n \mc{X}^{(k)} +  (\mc{D} + \mc{L})^{-1} *_N \mc{B}.
\end{eqnarray*}
The iteration scheme will convergent if $\rho([\mc{I} - (\mc{D} + \mc{L})^{-1} \n \mc{A}]) <  1$.
Let $\lambda$ be the eigenvalue of $\mc{I} - (\mc{D} + \mc{L})^{-1} \n \mc{A}$. Then $(\mc{I} - (\mc{D} + \mc{L})^{-1} \n \mc{A})\n\mc{X} = \lambda\mc{X}$, which implies 
\begin{eqnarray}\label{tseidel}
(\mc{D} + \mc{L})\n \mc{X} - \mc{A}\n \mc{X} = \lambda(\mc{D}+\mc{L})*_N\mc{X}.
\end{eqnarray}
This can be written in term of components,
\begin{equation}\label{myth}
    -\displaystyle\sum_{\textbf{j}(N) = \textbf{i}(N)+1}^{\mathbf{\hat{I}}(N)} a_{\textbf{i}(N),\textbf{j}(N)} x_{\textbf{j}(N)} = \lambda   \displaystyle\sum_{\textbf{j}(N) = 1}^{\mathbf{i}(N)} a_{\textbf{i}(N),\textbf{j}(N)} x_{\textbf{j}(N)},
\end{equation}
where $\textbf{j}(N)$ from $1$ to $\textbf{i}(N)$, indicates  $\{ j_1 =1, j_2=1, \cdots, j_N=1 \}$ to $ \{ j_1 =i_1, j_2=i_2, \cdots, j_N=i_N \}$. Similarly,  one can represent  $\textbf{j}(N)$, from $(\textbf{i}(N)+1)$ to $\mathbf{\hat{I}}(N)$.
Now Eq. \eqref{myth} can be written as 
\begin{equation*}
    \lambda a_{\textbf{i}(N), \textbf{i}(N)} x_{\textbf{i}(N)} = - \displaystyle\sum_{\textbf{j}(N) = \textbf{i}(N)+1}^{\mathbf{\hat{I}}(N)} a_{\textbf{i}(N),\textbf{j}(N)} x_{\textbf{j}(N)} - \lambda   \displaystyle\sum_{\textbf{j}(N)=1}^{\mathbf{{i}}(N)-1} a_{\textbf{i}(N),\textbf{j}(N)} x_{\textbf{j}(N)},
\end{equation*}
which equivalent to, 
\begin{equation}\label{eq5.17}
    \left|\lambda a_{\textbf{i}(N), \textbf{i}(N)} x_{\textbf{i}(N)} \right| \leq   \displaystyle\sum_{\textbf{j}(N) = \textbf{i}(N)+1}^{\mathbf{\hat{I}}(N)} \left|a_{\textbf{i}(N),\textbf{j}(N)}\right| \left|x_{\textbf{j}(N)}\right| + \left|\lambda\right|   \displaystyle\sum_{\textbf{j}(N) = 1}^{\textbf{i}(N)-1} \left|a_{\textbf{i}(N),\textbf{j}(N)}\right| \left| x_{\textbf{j}(N)}\right|.
\end{equation}
Without loss of generality, one can assume that $||\mc{X}||_\infty =1$. Choose indices $\textbf{i}(N)$ such that $|x_{\textbf{i}(N)}| = 1$ and $x_{\textbf{j}(N)} \leq 1$ for all $\textbf{i}(N) \neq \textbf{k}(N)$. We obtain from the above Eq. \eqref{eq5.17},
\begin{equation*}
    \left|\lambda \right| \left|a_{\textbf{i}(N), \textbf{i}(N)}\right| \leq   \displaystyle\sum_{\textbf{j}(N) = \textbf{i}(N)+1}^{\mathbf{\hat{I}}(N)} \left|a_{\textbf{i}(N),\textbf{j}(N)}\right|  + \left|\lambda\right|   \displaystyle\sum_{\textbf{j}(N) = 1}^{\textbf{i}(N)-1} \left|a_{\textbf{i}(N),\textbf{j}(N)}\right|, 
\end{equation*}
or 
\begin{equation*}
    \left|\lambda \right| \left[ \left|a_{\textbf{i}(N), \textbf{i}(N)}\right| - \displaystyle\sum_{\textbf{j}(N) = 1}^{\textbf{i}(N)-1} \left|a_{\textbf{i}(N),\textbf{j}(N)}\right| \right] \leq   \displaystyle\sum_{\textbf{j}(N) = \textbf{i}(N)+1}^{\mathbf{\hat{I}}(N)} \left|a_{\textbf{i}(N),\textbf{j}(N)}\right|. 
\end{equation*}
Since the tensor $\mc{A}$ is strictly diagonally dominant, the above inequality  becomes
\begin{equation*}
    \left|\lambda \right|  \leq   \frac{\displaystyle\sum_{\textbf{j}(N) = \textbf{i}(N)+1}^{\mathbf{\hat{I}}(N)} \left|a_{\textbf{i}(N),\textbf{j}(N)}\right|}{\left[ \left|a_{\textbf{i}(N), \textbf{i}(N)}\right| - \displaystyle\sum_{\textbf{j}(N) = 1}^{\textbf{i}(N)-1} \left|a_{\textbf{i}(N),\textbf{j}(N)}\right| \right]} < 1 .
\end{equation*}
 
\end{proof}
\end{theorem}
The idea behind iterative methods is to save memory and operational costs for solving multilinear systems. In light of this, the higher order Gauss-Seidel method is described in Algorithm 1. The application of this algorithm is illustrated in the following Poisson problem. 
\begin{small}
\begin{algorithm}[H]
\caption{Higher order Gauss-Seidel Method}\label{alg2}
\begin{algorithmic}[1]
\Procedure{Gauss-Seidel}{$\mc{A},\mc{B},\epsilon, \mbox{MAX}$}
\State Given $\mc{A}\in\mathbb{R}^{\textbf{I}(N) \times \textbf{I}(N)},~  \mc{B}\in\mathbb{R}^{\textbf{I}(N)},\mbox{ and } \mbox{MAX}$
\State Initial guess $\mc{X}^{(0)}\in\mathbb{R}^{\textbf{I}(N)}$
\NoDo
\For{$k=1$ to MAX}

\For{ $\textbf{i}(N)=1$ to $\hat{\textbf{I}}(N)$}
    \For{ $\textbf{j}(N)=1$ to $\hat{\textbf{J}}(N)$}
              \State \small{\hspace*{-1.5cm}
              {$(\mc{X}^{(k)})_{\textbf{i}(N)}=\displaystyle\frac{1}{{a}_{\textbf{i}(N), \textbf{i}(N)}}\left({b}_{\textbf{i}(N)}-\displaystyle\sum_{\textbf{j}(N)=1}^{\textbf{i}(N)-1}{a}_{\textbf{i}(N), \textbf{j}(N)}(\mc{X}^{(k)})_{\textbf{j}(N)}-\displaystyle\sum_{\textbf{j}(N)=\textbf{i}(N)+1}^{\hat{\textbf{J}}(N) }{a}_{\textbf{i}(N), \textbf{j}(N)}(\mc{X}^{(k-1)})_{\textbf{j}(N)}\right)$}}
    \EndFor
    \EndFor
    \If {$(\|\mc{X}^{(k)}-\mc{X}^{(0)}\|<\epsilon)$}
\Break
\EndIf
\State $\mc{X}^{(0)}\leftarrow \mc{X}^{(k)}$
\EndFor
\State \Return $\mc{X}^{(k)}$
\EndProcedure
\end{algorithmic}
\end{algorithm}
\end{small}
\begin{example}\label{example-2}
Consider the two-dimension Poisson problem
\begin{eqnarray*}
-\nabla^2 u&=&f(x,y)    \textnormal{~~in~~} \Omega \\
u&=& 0  \textnormal{~~on~~} \partial \Omega ,
\end{eqnarray*}
where $\Omega = \{ (x,y) : 0 < x, y < 1\}$ with boundary $\partial \Omega$, $f(x,y)$ is a given function, and 
\begin{equation*}
   \nabla^2u = \frac{\partial^2 u}{\partial x^2}+\frac{\partial^2 u}{\partial y^2}.
\end{equation*}
Here, our aim is to compute an approximation of the unknown function $u(x,y)$.
Using $5$-point stencil central difference scheme on a discretizing the unit square domain 
with $n$ interior nodes, we obtain a multilinear system
\begin{eqnarray*}
\mc{A}*_2\mc{X}=\mc{B}, \textnormal{~~ where~} \mc{A}\in\mathbb{R}^{n\times n\times n\times n}, ~\mc{X}\in\mathbb{R}^{n\times n},\textnormal{~and~} \mc{B}\in\mathbb{R}^{n\times n}.
\end{eqnarray*}
The tensor $\mc{A}$ is of the form 
\begin{eqnarray}\label{2nd}
\mc{A}=\mc{I}_n\kronecker \mc{P}_n+\mc{P}_n\kronecker \mc{I}_n,
\end{eqnarray}
where $\mc{I}_n$ is the second order identity tensor and $\mc{P}_n$ is  also a second order tensor  of the form 
$P_n=tridiagonal(-1, 2,-1).$
%
\begin{figure}[t!]\label{figure-2}
\begin{center}
\begin{tabular}{cc}
\subfigure[] { \includegraphics[width=0.4\textwidth]{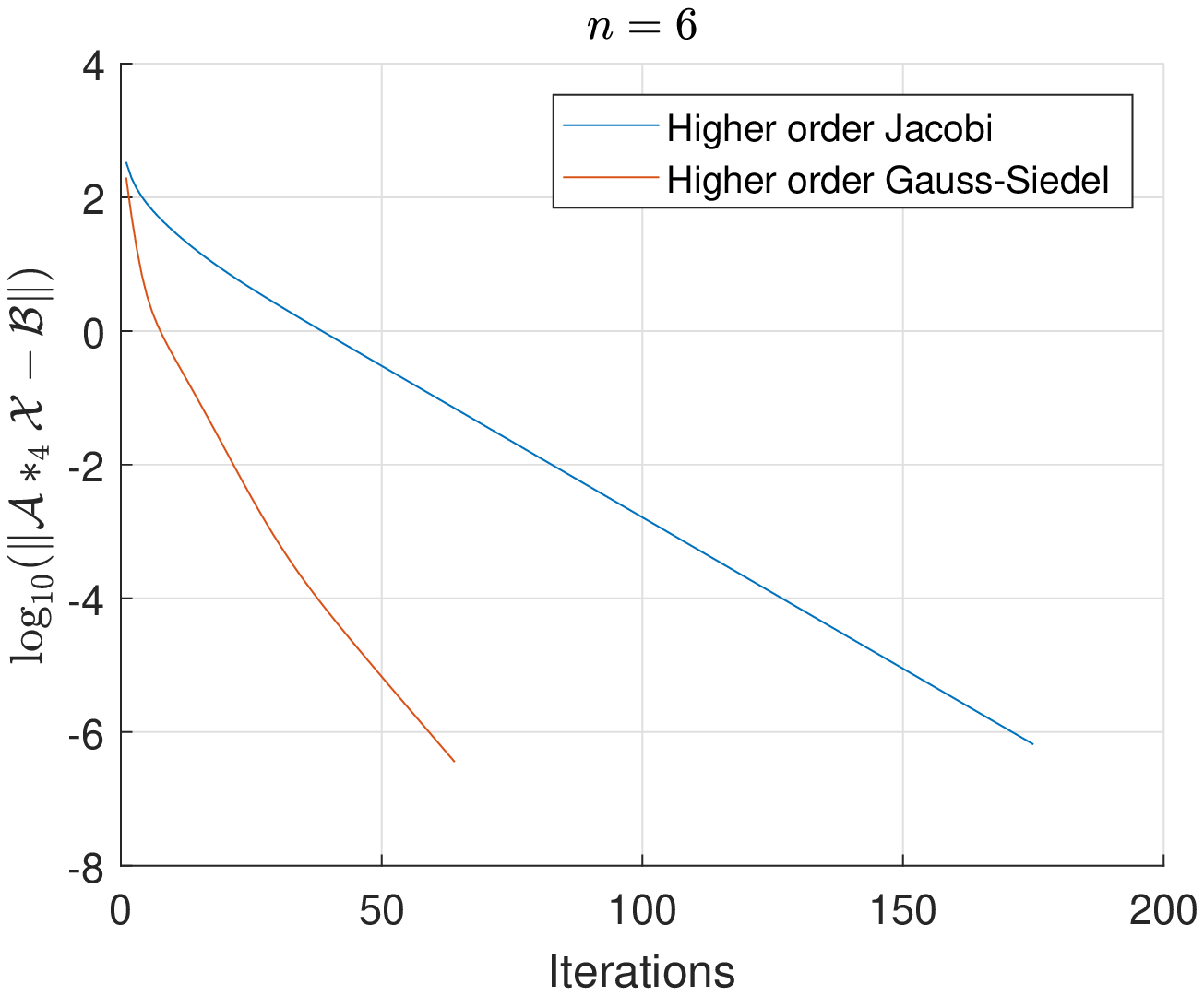}}&
\subfigure[] { \includegraphics[width=0.4\textwidth]{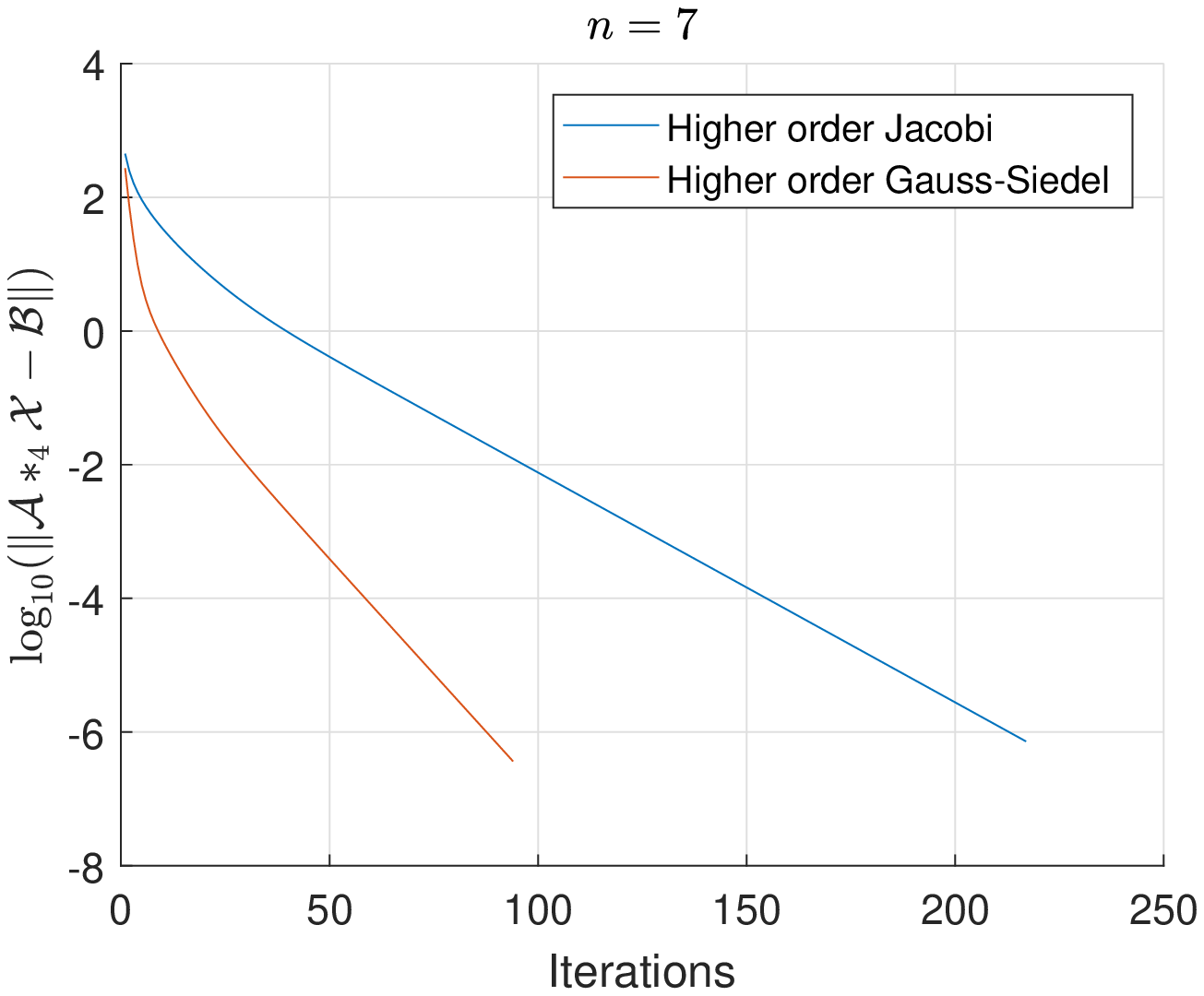}}\\
\subfigure[] { \includegraphics[width=0.4\textwidth]{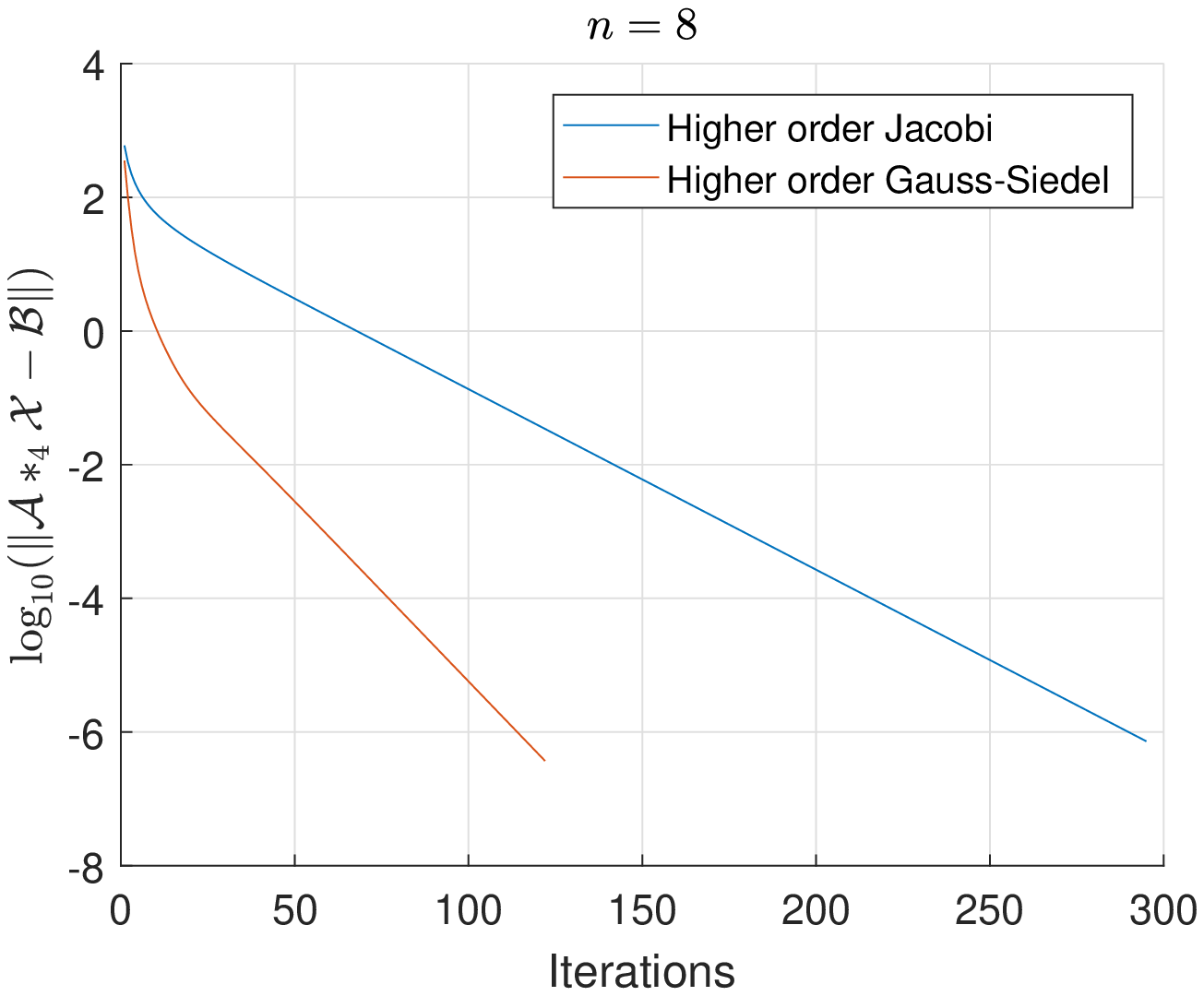}}&
\subfigure[] { \includegraphics[width=0.4\textwidth]{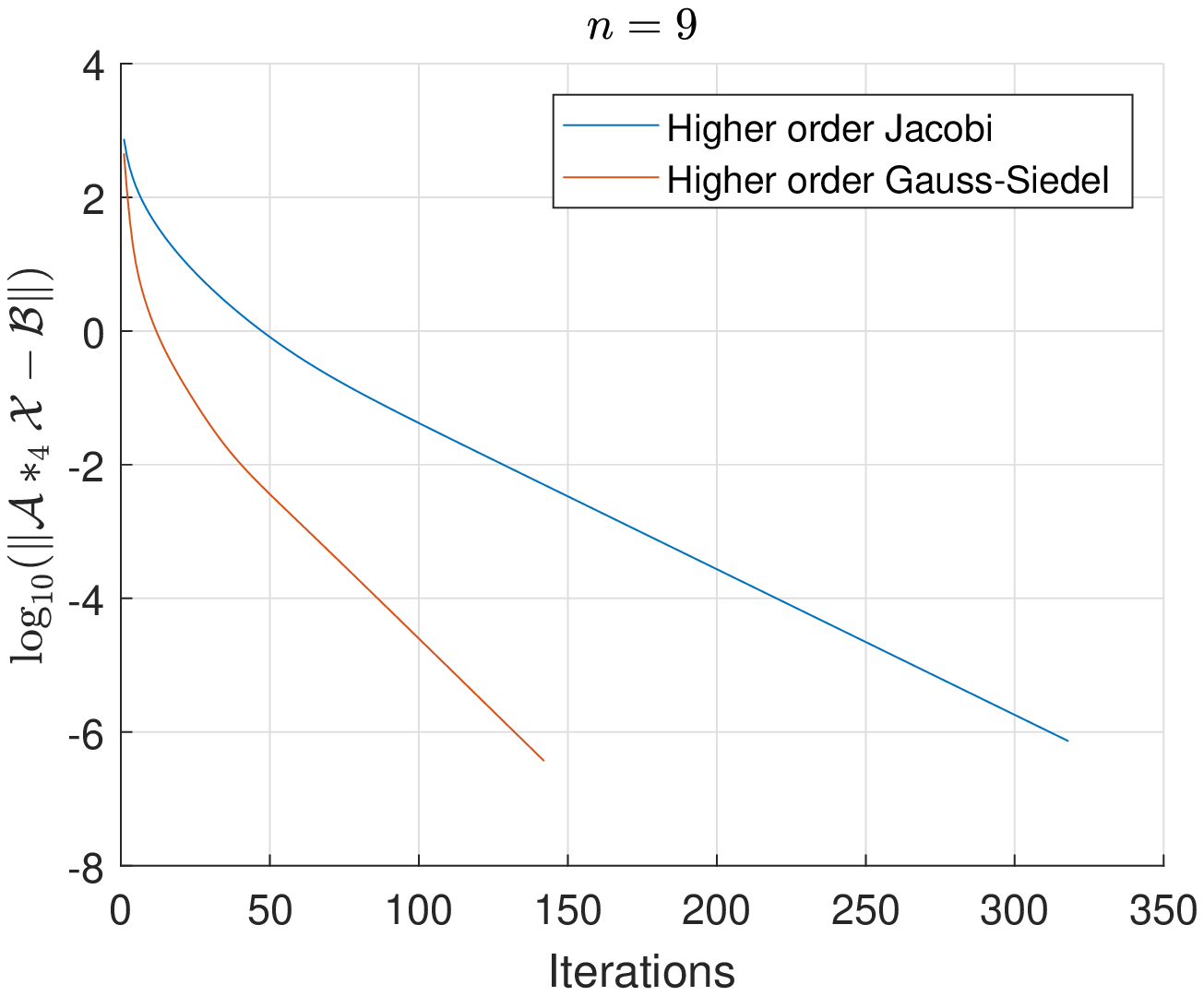}}
\end{tabular}
\caption{Residual error with respect to the number of iterations for different values of $n$.}
\label{exa-2}
\end{center}
\end{figure}
By applying 7-point stencil formula \cite{bral} for 3-dimensional Poisson equation with same boundary conditions, we obtain the following tensor equation 
\begin{equation*}
\mc{A}*_3\mc{X}=\mc{B}, ~ \mc{A}\in\mathbb{R}^{n\times n\times n\times n\times n\times n},~   \mc{X}\in\mathbb{R}^{n\times n\times n}, \mbox{ and } \mc{B}\in\mathbb{R}^{n\times n\times n},
\end{equation*}
where the tensor $\mc{A}$ is the following form 
\begin{eqnarray}\label{3rd}
 \mc{A}=\mc{P}_n\kronecker \mc{I}_n\kronecker \mc{I}_n + \mc{I}_n \kronecker\mc{P}_n\kronecker \mc{I}_n+ \mc{I}_n\kronecker \mc{I}_n\kronecker\mc{P}_n. 
 \end{eqnarray}
 Extending, the same idea to 4-dimensional Poisson problem, we obtain the following multilinear system 
\begin{equation}\label{4dp}
   \mc{A}*_4\mc{X}=\mc{B}, ~ \mc{A}\in\mathbb{R}^{\overbrace{n\times \cdots \times n}^{8\rm\ times}},~   \mc{X}\in\mathbb{R}^{n\times n\times n\times n}, \mbox{ and } \mc{B}\in\mathbb{R}^{n\times n\times n\times n}, 
\end{equation}
where the tensor $\mc{A}$ is the following form 
\begin{eqnarray}\label{4th}
\mc{A}=\mc{P}_n\kronecker \mc{I}_n\kronecker \mc{I}_n \kronecker \mc{I}_n+ \mc{I}_n \kronecker\mc{P}_n\kronecker \mc{I}_n\kronecker \mc{I}_n+ \mc{I}_n\kronecker \mc{I}_n\kronecker\mc{P}_n\kronecker \mc{I}_n+\mc{I}_n\kronecker \mc{I}_n+ \mc{I}_n\kronecker\mc{P}_n.
\end{eqnarray}
In the light of the above Eqs \eqref{2nd}, \eqref{3rd} and \eqref{4th} one can generate the tensor $\mc{A}$ to solve higher dimensional Poisson problem. In order to illustrate the efficiency of the proposed iterative method, we consider the four-dimensional Poisson problem \eqref{4dp} and a tensor $\mc{B}$ from $\mathfrak{R}(A)$. The residual error with respect to the number of iterations for different choices of $n$ are shown in Figure-2. In addition to this, the residual error is also compared with higher order Jacobi iterative method \cite{bral}.
\end{example}

\section{Conclusion}
We have discussed some more results on the Drazin inverse of tensors via the Einstein product to the existing theory. In particular, we have studied different characterizations of  the Drazin inverse and W-weighted Drazin inverse of tensors. Further, the concept of full rank decomposition is used to compute the group inverse of tensors and applications of these results discussed by solving multilinear systems. In addition to this, we have discussed the convergence analysis of the iterative technique.  The obtained results in this paper are important for the tensor splitting theory. During our study, we obtain a few sufficient conditions of the reverse-order law for the Drazin inverse of tensors.  However, various reverse-order laws for  the Drazin inverse of tensor formulas associated with rank and block-tensor  works are currently underway. \\
\section*{Acknowledgments}
R. Behera acknowledges the support provided by Science and Engineering Research Board (SERB), Department of Science and Technology, India, under the Grant No. MTR/2017/000417.
\bibliographystyle{abbrv}
\bibliographystyle{elsarticle-num}
\bibliography{DrazinRAJ}
\end{document}